\documentclass[a4paper,12pt,reqno]{amsart}

\usepackage[T1]{fontenc}
\usepackage[utf8]{inputenc}
\usepackage{lmodern}
\usepackage[english]{babel}

\usepackage[%
	left=2.5cm,       
	right=2.5cm,      
	top=3.5cm,        
	bottom=3.5cm,     
	heightrounded,    
	bindingoffset=0mm 
]{geometry}

\usepackage{amsmath}
\usepackage{amssymb}
\usepackage{mathtools}
\usepackage{mathrsfs}
\usepackage[normalem]{ulem}
\usepackage{hyperref} 
\usepackage[noabbrev,capitalize]{cleveref}
\usepackage{bookmark}
\usepackage[initials,
]{amsrefs}
\usepackage[dvipsnames]{xcolor}
\usepackage{graphicx}
\usepackage{url}

\usepackage{enumitem}

\numberwithin{equation}{section}

\theoremstyle{plain}
	\newtheorem{theorem}{Theorem}[section]
	\newtheorem{lemma}[theorem]{Lemma}
	\newtheorem{proposition}[theorem]{Proposition}
	\newtheorem{corollary}[theorem]{Corollary}

\theoremstyle{definition}
\newtheorem{definition}[theorem]{Definition}

\newtheorem{remark}[theorem]{Remark}

\newcommand{\N}{\mathbb{N}}
\newcommand{\R}{\mathbb{R}}
\newcommand{\eps}{\varepsilon}

\newcommand{\closure}[2][3]{%
  {}\mkern#1mu\overline{\mkern-#1mu#2}}

\renewcommand{\phi}{\varphi}
\renewcommand{\rho}{\varrho}
\renewcommand{\theta}{\vartheta}

\DeclareMathOperator{\supp}{supp}
\DeclareMathOperator{\loc}{{loc}}	    
\DeclarePairedDelimiter{\set}{\{}{\}}
\DeclarePairedDelimiter{\pairing}{\langle}{\rangle}

\def\Xint#1{\mathchoice
{\XXint\displaystyle\textstyle{#1}}%
{\XXint\textstyle\scriptstyle{#1}}%
{\XXint\scriptstyle\scriptscriptstyle{#1}}%
{\XXint\scriptscriptstyle\scriptscriptstyle{#1}}%
\!\int}
\def\XXint#1#2#3{{\setbox0=\hbox{$#1{#2#3}{\int}$ }
\vcenter{\hbox{$#2#3$ }}\kern-.6\wd0}}

\def\aint{\Xint-}

\DeclareMathOperator*{\essinf}{ess-inf}
\DeclareMathOperator*{\esssup}{ess-sup}
\DeclareMathOperator{\leb}{\mathscr{L}}

\newcommand{\dimension}{n}
\newcommand{\Rdim}{{\R^{\dimension}}}

\newcommand{\comp}[1]{{#1}^c} 

\newcommand{\Rkernel}{\Gamma_{\kernel}}
\newcommand{\Rkerneldual}{\Rkernel^{*}}
\newcommand{\BVkerdual}{BV^{-\kernel}}

\newcommand{\Gekeland}{{\mathcal{G}}}
\newcommand{\Aekeland}{\mathcal{A}}
\newcommand{\Xekeland}{{\mathbf{X}}}
\newcommand{\Yekeland}{{\mathbf Y}}

\newcommand{\bfx}{{\mathbf{x}}}
\newcommand{\bfy}{{\mathbf{y}}}

\newcommand{\measurablesets}{\mathcal{M}_{\dimension}}
\newcommand{\lebdim}{\leb^{\dimension}}
\newcommand{\lebone}{\leb^1}
\newcommand{\de}{\mathrm{d}}

\newcommand{\BVkerspace}{BV^{\kernel}}
\newcommand{\BVkerlocspace}{BV_{\loc}^{\kernel}}
\newcommand{\kernel}{K}
\newcommand{\kerTV}[1]{[#1]_{\BVkerspace(\Rdim)}}

\newcommand{\kerNablap}[1]{D _{\kernel,p}#1}
\newcommand{\Pker}{\mathcal{P}_{\kernel}}

\newcommand{\interaction}{L_{\dimension,\kernel}}

\newcommand{\Lspace}[1]{L^{#1}}
\newcommand{\Llocspace}[1]{L_{\loc}^{#1}}

\newcommand{\EnergyLOneKerTV}[1]{\mathcal{E}^{\kernel}#1}
\newcommand{\EnergyGeometricKer}[1]{\mathcal{E}^{\kernel}_{\mathbf{G}}#1}

\newcommand{\GeometricProblem}[2]{{GP}^{\kernel}_{#1,#2}}
\newcommand{\GeometricProblemSolutions}[1]{\boldsymbol{GSol}^{\kernel}\!\left(#1\right)}
\newcommand{\FunctionalProblem}[2]{{P}^{\kernel}_{#1,#2}}
\newcommand{\FunctionalProblemSolutions}[1]{\boldsymbol{Sol}^{\kernel}\!\left(#1\right)}
\newcommand{\FunctionalProblemJumps}[1]{\boldsymbol{J}^{\kernel}\!\left(#1\right)}

\newcommand{\distSolDatPLUS}{\mu_{\kernel}^{+}}
\newcommand{\distSolDatMINUS}{\mu_{\kernel}^{-}}
\newcommand{\distSolDatPLUSMINUS}{\mu_{\kernel}^{\pm}}

\newcommand{\Cheegker}{h_{\kernel}}
\newcommand{\CheegerkerClass}{\mathcal{C}_{\kernel}}
\DeclarePairedDelimiter{\tonde}{(}{)}

\newcommand{\peso}{w}
\newcommand{\misurapeso}{\nu}
\newcommand{\belowpeso}{\underline{\misurapeso}}
\newcommand{\abovepeso}{\overline{\misurapeso}}
\newcommand{\MisureAmmissibili}{\mathcal{W}(\Rdim)}

\newcommand{\Cheegkerpeso}{h_{\kernel,\misurapeso}}
\newcommand{\CheegerkerpesoClass}{\mathcal{C}_{\kernel,\misurapeso}}

\allowdisplaybreaks

\begin{document}

\title[Non-local $BV$ functions and a denoising model with $L^1$ fidelity]{Non-local $BV$ functions and a denoising model\\ with $L^1$ fidelity}

\author[K.~Bessas]{Konstantinos Bessas}
\address[K.~Bessas]{Dipartimento di Matematica, Università di Pavia, Via Adolfo Ferrata 5, 27100 Pavia (PV), Italy}
\email{konstantinos.bessas@unipv.it}

\author[G.~Stefani]{Giorgio Stefani}
\address[G.~Stefani]{Scuola Internazionale Superiore di Studi Avanzati (SISSA), via Bonomea 265, 34136 Trieste (TS), Italy}
\email{gstefani@sissa.it {\normalfont or} giorgio.stefani.math@gmail.com}
\date{\today}



\thanks{\textit{Acknowledgements}. 
The authors thank Matteo Novaga and Valerio Pagliari for several helpful comments, and Guy Fabrice Foghem Gounoue for indicating his Ph.D.\ thesis to us.
The authors	are members of the Istituto Nazionale di Alta Matematica (INdAM), Gruppo Nazionale per l'Analisi Matematica, la Probabilità e le loro Applicazioni (GNAMPA) and are partially supported by the INdAM--GNAMPA 2023 Project \textit{Problemi variazionali per funzionali e operatori non-locali}, codice CUP\_E53\-C22\-001\-930\-001.
The first-named author is partially supported by the INdAM--GNAMPA 2022 Project \textit{Fenomeni non locali in problemi locali}, codice CUP\_E55\-F22\-00\-02\-70\-001.
The second-named author is partially supported by the INdAM--GNAMPA 2022 Project \textit{Analisi geometrica in strutture subriemanniane}, codice CUP\_E55\-F22\-00\-02\-70\-001, and has received funding from the European Research Council (ERC) under the European Union’s Horizon 2020 research and innovation program (grant agreement No.~945655).
Part of this work was undertaken while the second-named author was visiting the University of Pisa.
He would like to thank the University of Pisa for its support and warm hospitality.}

\begin{abstract}
We study a general total variation denoising model with weighted $L^1$ fidelity, where the regularizing term is a non-local variation induced by a suitable (non-integrable) kernel $\kernel$, and the approximation term is given by the $\Lspace{1}$ norm with respect to a non-singular measure with positively lower-bounded $\Lspace{\infty}$ density. 
We provide a detailed analysis of the space of non-local $BV$ functions with finite total $\kernel$-variation, with special emphasis on compactness, Lusin-type estimates, Sobolev embeddings and isoperimetric and monotonicity properties of the $\kernel$-variation and the associated $\kernel$-perimeter.
Finally, we deal with the theory of Cheeger sets in this non-local setting and we apply it to the study of the fidelity in our model.
\end{abstract}

\maketitle


\newcommand{\red}[1]{\textcolor{red}{\underline{\textit{#1}}}}

\section{Introduction}


\subsection{Total variation denoising models}

Total variation minimizing models have been employed in a wide variety of image restoration
problems. 
The most common model is \emph{denoising}, that is, preserving the most significant features of an image while removing the background noise.

Total variation denoising models were introduced by Rudin, Osher and Fatemi (ROF) in their pioneering work~\cite{ROF92}. 
In a domain $\Omega\subset\Rdim$ (for instance, the computer screen), we are given a certain greyscale image, identified with its greyscale function $f\colon\Omega\to\R$, which is supposed to be the noise-corrupted version of a clearer initial image.
The idea of the ROF model is to \emph{denoise} the given image $f$ by looking for a new greyscale image $u\colon\Omega\to\R$ solving the minimization problem
\begin{equation}
\label{eqi:rof}
\inf_{u\in BV(\Omega)}
\int_\Omega d|Du| +\Lambda \int_\Omega|u-f|^2\,\de x.
\end{equation}  
Here  $BV(\Omega)$ is the space of functions with \emph{bounded variation} on $\Omega$ and the Lagrangian multiplier $\Lambda>0$ is the \emph{fidelity} parameter.

Motivated by the lack of \emph{contrast invariance} (i.e., homogeneity) of~\eqref{eqi:rof}, Chan and Esedo\=glu (CE) in~\cite{CE05} proposed the alternative model
\begin{equation}
\label{eqi:ce}
\inf_{u\in BV(\Rdim)}
\int_\Rdim d|Du| +\Lambda \int_\Rdim |u-f|\,\de x.
\end{equation}  
The CE model~\eqref{eqi:ce} has the advantage to be contrast invariant, but it is  merely convex, thus losing uniqueness of minimizers. 
Nonetheless, the CE model~\eqref{eqi:ce} has a strongly geometric flavor, yielding an interesting link with the \emph{Cheeger constant}~\cites{L15,P11,FPSS22}.

The models~\eqref{eqi:rof} and~\eqref{eqi:ce} are of local nature, meaning that the regularizing term is of \emph{local} type.
Local denoising models are quite efficient, but they scarcely preserve the details, since fine structures are mostly treated as noise and thus smoothed out. 

For these reasons, several works~\cites{BCM10,KOJ05,GO07,GO08} turned the attention towards suitable \emph{non-local} replacements of the $BV$ energy.
A main advantage of non-local energies is that they consider both geometric parts and textures within the same framework, sharpening the sensibility of the regularization term.

In~\cite{MST22}, Maz\'on, Solera and Toledo  studied the ROF and CE models~\eqref{eqi:rof} and~\eqref{eqi:ce} in \emph{metric random walk spaces}.
In $\R^n$, their non-local regularization term becomes
\begin{equation}
\label{eqi:mazon}
u\mapsto
\frac12
\int_\Rdim
\int_\Rdim
|u(x)-u(y)|\,\kernel(x-y)\,\de x\,\de y,
\end{equation}
where 
$\kernel$
is a non-negative and radially symmetric function satisfying
$\int_\Rdim\kernel\,\de x=1$.
A detailed analysis of the energy~\eqref{eqi:mazon} has been carried out in~\cites{MRT19-a,MRT19-b}.
Since $\kernel\in\Lspace{1}(\Rdim)$, the energy~\eqref{eqi:mazon} is finite for any $u\in\Lspace{1}(\Rdim)$, hence $\Lspace{1}$ data are denoised to $\Lspace{1}$ images.

To improve the regularity of the denoised image, one must drop the assumption $\kernel\in\Lspace{1}(\Rdim)$.
In this direction, Novaga and Onoue~\cite{NO21} and  the first-named author~\cite{B22} studied the ROF model~\eqref{eqi:rof} and the CE model~\eqref{eqi:ce}, respectively, for $
\kernel=|\cdot|^{-\dimension-s}$ with $s\in(0,1)$.
We also refer to the recent work~\cites{ADJS22}, where the authors also study the so-called \emph{distributional fractional $s$-variation}, $s\in(0,1)$, introduced by Comi and the second-named author in~\cite{CS19} (also see~\cites{BCCS22,CS19,CS22-a,CS22-f,CS22-l,CSS22,CS23-dm,CS23-on} for an extensive treatment of this non-local energy).

As in the local case, also non-local ROF and CE models can be naturally linked with the theory of Cheeger sets, see~\cite{MST22}*{Sec.~3.2} and~\cite{B22}*{Sec.~4}.

\subsection{Main results}

The main aim of the present paper is to study the CE model \eqref{eqi:ce} with non-local regularization term of the form~\eqref{eqi:mazon} for non-integrable kernels satisfying minimal assumptions.
Our approach not only covers~\cite{B22}, but also allows for other non-local energies already appearing in the literature~\cites{CGO08,CD15,BN20,BN21,DGM04}, see \cref{subsec:examples_kernels} below.

The core of our approach is a detailed study of the space
\begin{equation*}
\BVkerspace(\Rdim)
=
\set*{u\in\Lspace{1}(\Rdim) : 
\kerTV{u}=\frac12
\int_\Rdim
\int_\Rdim
|u(x)-u(y)|\,\kernel(x-y)\,\de x\,\de y<+\infty
}
\end{equation*}
naturally induced by the seminorm~\eqref{eqi:mazon}.
The space $BV^\kernel(\R^n)$ has attracted considerable attention in recent years,
mostly in view of \emph{non-local minimal surfaces} and \emph{calibrations}~\cites{C20,CSV19,P20}, \emph{isoperimetric inequalities}~\cites{CN18,L14,K21},
\emph{non-local mean curvature flows}~\cites{CN22,CMP15}, \emph{asymptotic properties}~\cites{BP19} and \emph{embedding theorems}~\cites{F20,F21,JW20,CdP20}.

Besides improving the results available in the literature concerning basic properties of $\BVkerspace$ functions, we prove
\emph{compactness} of the embedding $\BVkerspace(\Rdim)\subset\Llocspace{1}(\Rdim)$,
a \emph{Lusin-type estimate} for $\BVkerspace$ functions,
a \emph{rearrangement inequality} and a \emph{Gagliardo--Nirenberg--Sobolev-like embedding},
a \emph{monotonicity property} of an isoperimetric-type ratio for dilations of a fixed set,
\emph{Sobolev-type inequalities}
and 
the
\emph{monotonicity} of the $\kernel$-perimeter on finite-measure sets under intersection with convex sets.

Concerning Cheeger sets in the $\BVkerspace$ framework, we prove 
\emph{existence} and \emph{basic properties} of Cheeger sets,
\emph{calibrability of balls},
a \emph{Faber--Krahn-type inequality},
the \emph{relation} between Cheeger sets and $\BVkerspace$ eigenfunctions;
an \emph{$\Lspace{\infty}$ bound} for $\BVkerspace$ eigenfunctions 
and
a non-local formulation of the \emph{Max Flow Min Cut Theorem}. 

Finally, we study the following non-local analogue of the CE model~\eqref{eqi:ce} 
\begin{equation}
\label{eqi:bs}
\inf_{u\in\BVkerspace(\Rdim)}
\kerTV{u}
+
\Lambda\int_\Rdim|u-f|\,\de\misurapeso.
\end{equation}
In our model~\eqref{eqi:bs}, the $\Lspace{1}$ approximation term is of \emph{weighted} type, namely, we integrate with respect to an \emph{admissible weight measure} $\misurapeso\in\MisureAmmissibili$, where
\begin{equation*}
\MisureAmmissibili
=
\set*{
\misurapeso=\peso\lebdim :
w\in L^\infty(\Rdim),\  \essinf\limits_{\R^n} w>0
}.
\end{equation*}

Besides the properties of the solutions of~\eqref{eqi:bs}, and their link with the solutions of the \emph{geometric} analog of~\eqref{eqi:bs}, we prove 
\emph{existence} of solutions for $\Lspace{1}$ data,
existence of (unique) \emph{minimal} and \emph{maximal} geometric solutions, 
a \emph{comparison principle} for (extremal) geometric solutions,
\emph{uniqueness} of geometric solutions for a.e.\ $\Lambda>0$ for bounded convex data,
\emph{high-fidelity results} for sufficiently smooth data,
a \emph{low-fidelity result} for $\Lspace{1}$ data with bounded support
and
the \emph{relation} between the fidelity and the Cheeger constant.
 
\subsection{Organization of the paper}
The rest of the paper is organized as follows.
\cref{sec:non-local_BV_func} is dedicated to the study of $\BVkerspace$ functions and sets.
\cref{sec:cheeger} deals with the theory of Cheeger sets in the $\BVkerspace$ framework.
\cref{sec:functional_pb,sec:geometric_pb} deal with the CE model~\eqref{eqi:bs} and its geometric analog, respectively.
Finally, \cref{sec:fidelity} studies the behavior of solutions as the fidelity parameter varies, proving high and low fidelity results.

\section{Non-local \texorpdfstring{$BV$}{BV} functions}
\label{sec:non-local_BV_func}

In this section, we study the non-local variation and non-local perimeter functionals.

\subsection{Assumptions on the kernel}

\label{subsec:assumptions}

We call a \emph{kernel} $\kernel\colon\Rdim\to[0,+\infty]$ any measurable function such that $\kernel\not\equiv0$ (up to negligible sets).
Throughout the paper, a kernel $\kernel$ may satisfy some of the assumptions listed below.

We may require $\kernel$ to be \emph{symmetric},
\begin{equation}
\label{H:Symmetric} 
\tag{Sym}
\text{$\kernel(x)=\kernel(-x)$ for all $x\in\Rdim$},
\end{equation}
or \emph{radial}, if there exists a measurable \emph{profile} $\kappa\colon[0,+\infty)\to[0,+\infty]$ such that
\begin{equation}
\tag{Rad}
\label{H:Radial}
\text{
$\kernel(x)=\kappa(|x|)$ for all\ $x\in\Rdim$.
}
\end{equation}
Note that \eqref{H:Symmetric} is not truly relevant, since one may replace $\kernel$ with its symmetrized $\tilde\kernel(x)=\frac{\kernel(x)+\kernel(-x)}{2}$, $x\in\Rdim$.
We may reinforce~\eqref{H:Radial} as
\begin{equation}
\label{H:Radial_strict}
\tag{Rad$^+$}
\text{\eqref{H:Radial} holds with $\kappa$ strictly decreasing in a neighborhood of the origin.
}
\end{equation}

We may require $\kernel$ to be \emph{integrable far from the origin},
\begin{equation}
\label{H:Far_from_zero_int}
\tag{Far}
\text{
$\kernel\in\Lspace1(\Rdim\setminus B_r)$ for all $r>0$,
}
\end{equation}
or \emph{not too singular} (also known as \emph{$1$-Lévy property}),  
\begin{equation}
\label{H:Not_too_singular}
\tag{Nts}
\int_{\Rdim}(1\wedge|x|)\,\kernel(x)\,\de x<+\infty.
\end{equation}
Differently from~\cites{MRT19-a,MRT19-b}, we usually deal with \emph{non-integrable} kernels,
\begin{equation}
\label{H:Not_integrable}
\tag{Nint}
\kernel\notin\Lspace{1}(\Rdim).
\end{equation}

We may require $\kernel$ to be \emph{strictly positive}, 
\begin{equation}
\label{H:Positive}
\tag{Pos}
\text{
$K(x)>0$ for all\ $x\in\Rdim$.
}
\end{equation}
We may locally reinforce \eqref{H:Positive} via the \emph{positivity of the infimum of $\kernel$ around the origin},
\begin{equation}
\label{H:Inf}
\tag{Inf}
\text{
$\exists\,r,\mu>0$ : $\kernel(x)\ge\mu$ for all\ $x\in B_r$.
}
\end{equation}

We may require $\kernel$ to be \emph{decreasing of exponent $q\in[0,+\infty)$}, 
\begin{equation}
\label{H:Decreasing_q}
\tag{Dec$_q$}
\text{
$|x|\le |y|$ $\implies$ $ \kernel(y)|y|^q\le\kernel(x)|x|^q$
}
\end{equation}
or \emph{(locally) doubling of radius $D>0$}
\begin{equation}
\label{H:Doubling}
\tag{Dou$_D$}
\text{
$\exists\,C>0$ : $|y|=2|x|$, $|x|\le 2D$ $\implies$ $\kernel(x)\le C\kernel(y)$.
}
\end{equation}
Sometimes, we may need to reinforce \eqref{H:Decreasing_q} in the case $q=\dimension$ as 
\begin{equation}
\label{H:Decreasing_dim_strinct}
\tag{Dec$_\dimension^+$}
\text{
$|x|< |y|$ $\implies$ $ \kernel(y)|y|^\dimension<\kernel(x)|x|^\dimension$.
}
\end{equation}
If $\kernel$ satisfies \eqref{H:Decreasing_q}, we call $q\in[0,+\infty)$ the \emph{decreasing exponent} of~$\kernel$. 
If $\kernel$ satisfies~\eqref{H:Doubling} with $D=+\infty$, then we call the (smallest) constant $C>0$  in~\eqref{H:Doubling} the \emph{doubling constant} of $\kernel$ and $p=\log_2 C$ the \emph{doubling exponent} of $\kernel$.

\begin{remark}[On a generalization of \eqref{H:Decreasing_q}]
Assumption \eqref{H:Decreasing_q} may be generalized in many ways. 
For example, one can replace \eqref{H:Decreasing_q} with 
\begin{equation}
\label{H:Decreasing_general}
\tag{Dec-$\omega$}
|x|\le|y|
\implies
\kernel(y)\,\omega(|y|)
\le 
\kernel(x)\,\omega(|x|)
\end{equation} 
for some increasing function $\omega\colon[0,+\infty)\to[0,+\infty)$.
The more general \eqref{H:Decreasing_general} may be useful when $\kernel$ satisfies \eqref{H:Decreasing_q} for all $q\in[0,q_0)$ but not for the limiting $q=q_0$.
This is in fact the case when $\kernel$ is a radial function of mixed power-logarithmic order, as in the example~\eqref{eq:example_fraclog} in \cref{subsec:examples_kernels} below.  
We will not deal with the consequences of \eqref{H:Decreasing_general}.
\end{remark}

\begin{remark}[Assumptions outside negligible sets]
The assumptions listed above are given pointwise everywhere just to keep the exposition simple.
In fact, the same assumptions may hold only outside a $\lebdim$-negligible set in $\Rdim$, without affecting the validity of the results of this paper (up to the routine adaptations in the statements, if needed).
\end{remark}

\subsection{Relations among the assumptions}

\label{subsec:relations_ass}

We observe that
\eqref{H:Radial}$\implies$\eqref{H:Symmetric}
and 
\eqref{H:Not_too_singular}$\implies$\eqref{H:Far_from_zero_int}.
Also, if~\eqref{H:Decreasing_q} holds for some $q\ge0$, then it also holds for all $q'\in[0,q]$.
In particular,
\eqref{H:Radial} and \eqref{H:Decreasing_q} with $q>0$$\implies$\eqref{H:Radial_strict}.
We also have
\eqref{H:Decreasing_q} 
and
\eqref{H:Doubling} with $D=+\infty$$\implies$\eqref{H:Positive}.
Actually, one can state the following quantitative version of the latter implication in terms of the parameter $D\in(0,+\infty]$ in \eqref{H:Doubling}.

\begin{lemma}[Positivity]
\label{res:pos}
If $\kernel$ satisfies~\eqref{H:Decreasing_q} and~\eqref{H:Doubling} with $D>0$, then $\kernel(x)>0$ for all $x\in B_{4D}$.
In particular, if $D=+\infty$, then $\kernel$ satisfies~\eqref{H:Positive}.
\end{lemma}

Also the implication $\eqref{H:Decreasing_q}\implies\eqref{H:Inf}$ can be made quantitative depending on the value of $q\in[0,+\infty)$, as follows. 

\begin{lemma}[Comparison via \eqref{H:Decreasing_q}]
\label{res:dec_q_comparison}
If $\kernel$ satisfies~\eqref{H:Decreasing_q}, then
\begin{equation}
\label{eq:Dec_q_comparison}
\text{
$\exists\,R,C>0$ : $C\,\frac {\chi_{B_R}(x)}{|x|^q}\le\kernel(x)\le C\,\frac{\chi_{B_R^c}(x)}{|x|^q}+
\chi_{B_R}(x)\,\kernel(x)$ for all $x\in\Rdim\setminus\set*{0}$.
}
\end{equation}
In particular, $\kernel$ satisfies~\eqref{H:Inf}.
In addition, if $\kernel$ also satisfies~\eqref{H:Not_too_singular}, then $q<\dimension+1$; if, instead, $\kernel\in\Lspace{1}_{\loc}(\Rdim)$, then $q<\dimension$. 
\end{lemma}

Finally, the following result can be considered as a dual formulation of the previous \cref{res:dec_q_comparison} and is a simple consequence of the decaying property induced by \eqref{H:Doubling}.

\begin{lemma}[Comparison via \eqref{H:Doubling}]
\label{res:dou_comparison}
If $\kernel$ satisfies
\eqref{H:Decreasing_q} and
\eqref{H:Doubling} 
with $D=+\infty$, 
then
\begin{equation*}
\text{
$\forall R>0\ \exists\,m_R,M_R>0$ : 
$m_R\,\frac{\chi_{B_R^c}(x)}{|x|^p}\le\kernel(x)\le M_R\,\frac{\chi_{B_R}(x)}{|x|^p}+\chi_{B_R^c}(x)\,\kernel(x)$ 
}
\end{equation*}
for all $x\in\Rdim\setminus\set*{0}$,
where $p=\log_2C\ge q$ and $C\ge 2^q$ is the doubling constant of $\kernel$.
In addition, if $\kernel$ also satisfies~\eqref{H:Far_from_zero_int}, then $p>\dimension$, that is, \eqref{H:Doubling} holds with doubling constant $C\ge\max\set*{2^\dimension,2^q}$.
\end{lemma}

Lemmas~\ref{res:pos}, \ref{res:dec_q_comparison} and~\ref{res:dou_comparison} above follow from elementary arguments that are omitted. 

\subsection{Examples of kernels}
\label{subsec:examples_kernels}

The typical kernel our theory can be applied to is given by
\begin{equation}\label{eq:example_kernel}
\kernel_{\upsilon,\dimension}(x)
=
\frac{\upsilon(|x|)}{|x|^\dimension},
\quad
x\in\Rdim\setminus\set{0},
\end{equation}
where $\upsilon\colon(0,+\infty)\to[0,+\infty]$ is a measurable function. 
We set $\kernel_{\upsilon,\dimension}(0)=+\infty$. 
The kernel in~\eqref{eq:example_kernel} satisfies \eqref{H:Radial}, and the assumptions in \cref{subsec:assumptions} can be given in terms of~$\upsilon$.

With a slight abuse of notation, the space linked to~\eqref{eq:example_kernel} is 
\begin{equation*}
BV^{\kernel_{\upsilon,\dimension}}(\Rdim)
=
\set*{u\in\Lspace{1}(\Rdim)
:
(x,y)
\mapsto
\frac{|u(x)-u(y)|}{|x-y|^\dimension}
\,
\upsilon(|x-y|)
\in
\Lspace{1}(\Rdim\times\Rdim)
}
\end{equation*}
similarly as in~\cite{DGM04}*{Sec.~3} (although with a different notation).
The space $BV^{\kernel_{\upsilon,\dimension}}(\Rdim)$ is a generalization of the usual \emph{fractional Ga\-gli\-ar\-do--Slo\-bo\-deckij--So\-bo\-lev space} $W^{s,1}(\Rdim)$, which, in turn, corresponds to the choice 
$\upsilon(\rho)=\rho^{-s}$ for $\rho>0$,
for some fixed $s\in(0,1)$, see~\cite{DPV12} for an account.
All assumptions above are satisfied, with decreasing and doubling exponents $q=p=\dimension+s$.
More generally, one can consider 
\begin{equation*}
\upsilon(\rho)
=
\begin{cases}
\rho^{-s_0}
& 
\text{for}\ 0<\rho\le1,
\\[2mm]
\rho^{-s_1}
& 
\text{for}\ \rho>1,
\end{cases}
\end{equation*}
for fixed $0\le s_0\le s_1\le1$.
Also in this case, all assumptions are satisfied, with decreasing exponent $q=\dimension+s_0$ and doubling exponent $p=\dimension+s_1$.
One can in fact consider 
$\upsilon(\rho)
=
\rho^{-s(\rho)}
=
e^{-s(\rho)\log\rho}$
for $\rho>0$,
for some fractional-exponent (increasing) function $s\colon(0,+\infty)\to[0,1]$. 
One can also consider logarithms~\cites{BN20, DGM04,CD15,CGO08}, that is, letting
\begin{equation}
\label{eq:example_log}
\upsilon(\rho)
=
\begin{cases}
(1-\log\rho)^{-\alpha}
&
\text{for}\ 0<\rho\le 1,
\\[2mm]
0
& 
\text{for}\ \rho>1,
\end{cases}
\end{equation}
for fixed $\alpha\in\R$ or, more generally,
\begin{equation}
\label{eq:example_fraclog}
\upsilon(\rho)
=
\begin{cases}
\rho^{-s}\,(1-\log\rho)^{-\alpha}
&
\text{for}\ 0<\rho\le 1,
\\[2mm]
0
& 
\text{for}\ \rho>1,
\end{cases}
\end{equation} 
for fixed $s\in[0,1)$ and $\alpha\in\R$.
Obviously, in both~\eqref{eq:example_log} and~\eqref{eq:example_fraclog}, \eqref{H:Positive} fails, \eqref{H:Doubling} holds for $D<\frac12$ and \eqref{H:Far_from_zero_int}, \eqref{H:Not_too_singular} and \eqref{H:Inf} are true.
Moreover, \eqref{H:Not_integrable} holds if and only if
\begin{equation*}
(s,\alpha)\in(0,1)\times\R
\cup
\set*{0}\times(-\infty,1].
\end{equation*}
Finally, \eqref{H:Decreasing_q} holds for any $q\in[0,\dimension+s)$, and for $q=\dimension+s$ whenever $\alpha\in(-\infty,0]$.
Clearly, the examples in~\eqref{eq:example_log} and~\eqref{eq:example_fraclog} can be varied in many ways, in particular modifying the behavior of $\rho\mapsto\upsilon(\rho)$ for $\rho>1$ in order to also ensure  \eqref{H:Positive} and \eqref{H:Doubling} with $D=+\infty$ (and even the continuity of $\upsilon$) while keeping valid all the other properties.
Further examples can be also derived from the theory of \emph{Lévy kernels}~\cites{F20,F21,FKV20}.

\subsection{\texorpdfstring{$\kernel$}{K}-variation}

\label{subsec:variation}

Let $\kernel\colon\Rdim\to[0,+\infty]$ be a kernel.  
We let 
\begin{equation}
\label{eq:def_kerTV}
\kerTV{u}
=
\frac{1}{2}\int_{\Rdim}\int_{\Rdim}|u(x)-u(y)|\,\kernel(x-y)\,\de x\,\de y
\end{equation}
be the \emph{$\kernel$-variation} of $u\in\Llocspace{1}(\Rdim)$ and\begin{equation*}
\BVkerlocspace(\Rdim)
=
\set*{u\in \Llocspace{1}(\Rdim) : \kerTV{u}<+\infty} 
\end{equation*}
and $\BVkerspace(\Rdim)=\BVkerlocspace(\Rdim)\cap \Lspace{1}(\Rdim)$.
The space $\BVkerspace(\Rdim)$ is  Banach with respect to
\begin{equation*}
\|u\|_{\BVkerspace(\Rdim)}
=
\|u\|_{\Lspace{1}(\Rdim)}+\kerTV{u},
\quad
u\in \Lspace{1}(\Rdim).
\end{equation*} 
Since $\kernel\not\equiv0$, if $\kerTV{u}=0$ for $u\in\Llocspace{1}(\Rdim)$, then $u=c$ a.e.\ for some $c\in\R$, with $c=0$ if $u\in\Lspace{1}(\Rdim)$.
Note that $\BVkerspace(\Rdim)=\Lspace{1}(\Rdim)$ whenever $K\in \Lspace{1}(\Rdim)$, since
\begin{equation}
\label{eq:L1_ker_brutto}
\kerTV{u}
\le
\|K\|_{\Lspace{1}(\Rdim)}\,\|u\|_{\Lspace{1}(\Rdim)}
\end{equation}
by definition.
In addition,
\begin{equation}
\label{eq:valerio_embed}
\eqref{H:Not_too_singular}
\implies
BV(\Rdim)\subset\BVkerspace(\Rdim)\ 
\text{continuously},
\end{equation}
with
\begin{equation}
\label{eq:valerio}
\kerTV{u}
\le
\max\set*{\|u\|_{\Lspace{1}(\Rdim)},\tfrac12[u]_{BV(\R^n)}}
\int_{\Rdim}(1\wedge|x|)\,\kernel(x)\,\de x
\end{equation}
for all $u\in BV(\Rdim)$.
For a proof, see~\cite{BP19}*{Prop.~2.2} and~\cite{F-PhD}*{Lem.~3.48}.

\subsection{\texorpdfstring{$\kernel$}{K}-perimeter}

\label{subsec:perimeter}
For $E,A\in\measurablesets$, where $\measurablesets$ denotes $\lebdim$-measurable sets,
we let
\begin{equation}
\label{eq:def_Pker_relative}
\Pker(E;A)
=
\left(\frac12\int_A\int_A
+
\int_A\int_{A^c}\right)
|\chi_E(x)-\chi_E(y)|\,K(x-y)\,\de x\,\de y
\end{equation}
be the \emph{$\kernel$-perimeter of $E$ relative to $A$.}
In particular, for $A=\Rdim$ we let
\begin{equation*}
\Pker(E)
=
\Pker(E;\Rdim)
=
\kerTV{\chi_E}
\end{equation*}
be the \emph{(total) $\kernel$-perimeter of $E$}, so that $\Pker(E)=0$ if and only if $|E|=0$ or $|\Rdim\setminus E|=0$. 
If~\eqref{H:Symmetric} holds, then~\eqref{eq:def_Pker_relative} becomes
\begin{equation*}
\Pker(E;A)
=
\left(
\int_{E\cap A}\int_{\comp{E}\cap A}
+
\int_{E\cap A}\int_{\comp{E}\cap\comp{A}}
+
\int_{E\cap\comp{A}}\int_{\comp{E}\cap A}
\right)\,K(x-y)\,\de x\,\de y.
\end{equation*}
The $\kernel$-perimeter is invariant by translations, 
\begin{equation}
\label{eq:P_K_translation}
\Pker(E-x;A-x)=\Pker(E;A)
\end{equation}
for all $x\in\Rdim$, where $E-x=\set{y\in\Rdim : y+x\in E}$, it is invariant under complementation,
\begin{equation}
\label{eq:complement_invariance}
\Pker(E^c;A)=\Pker(E;A),
\end{equation}
and it satisfies the \emph{submodularity property}
\begin{equation}
\label{eq:submodularity}
\Pker(E\cap F;A)
+
\Pker(E\cup F;A)
\le
\Pker(E;A)
+
\Pker(F;A),
\end{equation}
see~\cite{BP19}*{p.~8, proof of~(vi)}. 
As in~\eqref{eq:L1_ker_brutto}, if $K\in\Lspace{1}(\Rdim)$ and~\eqref{H:Symmetric} holds, then 
\begin{equation}
\label{eq:L1_ker_perimeter}
\Pker(E)
=
\|K\|_{\Lspace{1}(\Rdim)}
|E|
-
\int_E\int_EK(x-y)\,\de x\,\de y
\end{equation}
provided that $|E|<+\infty$.
Finally, if~\eqref{H:Not_too_singular} is in force, then~\eqref{eq:valerio} becomes
\begin{equation}
\label{eq:valerio_set}
\Pker(E)
\le
\max\set*{|E|,\tfrac12P(E)}
\int_{\Rdim}(1\wedge|x|)\,\kernel(x)\,\de x,
\end{equation}
where $P(E)=[\chi_E]_{BV(\Rdim)}$ denotes the  \emph{perimeter} of $E$.

\subsection{Recognizing constant functions}

\label{subsec:constants}

We recall the following simple but fundamental result on how the $K$-variation allows to recognize constant functions~\cites{B02}.

\begin{proposition}[Constant functions]
\label{res:constants}
If the lower bound in~\eqref{eq:Dec_q_comparison} holds for  \mbox{$q\ge\dimension+1$}, then $\BVkerlocspace(\Rdim)$ contains a.e.-constant functions only, and $\BVkerspace(\Rdim)=\set*{0}$.
\end{proposition}

Due to \cref{res:dec_q_comparison} and the above \cref{res:constants},  \eqref{H:Decreasing_q} becomes truly relevant only for $q<\dimension+1$, which is the case if~\eqref{H:Not_too_singular} is in force, see the last part of~\cref{res:dec_q_comparison}.

If~\eqref{H:Decreasing_q} holds with $q>n$, then $\BVkerspace$ functions are fractional Sobolev of order $q-n$.

\begin{proposition}[Embedding in fractional Sobolev space]
\label{res:fractional_embed}
Let 
\eqref{H:Decreasing_q}
with $q\in(\dimension,\dimension+1)$
be in force.
Then $\BVkerspace(\Rdim)\subset W^{q-\dimension,1}_{\loc}(\Rdim)$ with continuous inclusion.
If, in addition, \eqref{H:Far_from_zero_int} holds, then $\BVkerspace(\Rdim)\subset W^{q-\dimension,1}(\Rdim)$ with continuous inclusion.
\end{proposition}

Propositions~\ref{res:constants} and~\ref{res:fractional_embed} above follow from elementary arguments that are omitted.

\subsection{Coarea formula}

\label{subsec:coarea}

The $\kernel$-variation and the $\kernel$-perimeter are connected via the following coarea formula, see~\cite{CN18}*{Prop.~2.3}.

\begin{lemma}[Coarea formula]
\label{res:coarea}
If $u\in \Llocspace{1}(\Rdim)$, then
$
\displaystyle
\kerTV{u}
=
\int_\R \Pker(\set*{u>t})\,\de t. 
$
\end{lemma}

\subsection{Lower semicontinuity}

\label{subsec:lsc}

Fatou's Lemma yields the semicontinuity property for the $\kernel$-variation and the $K$-perimeter, see~\cite{CN18}*{Prop.~2.2}.

\begin{lemma}[Lower semicontinuity in $\BVkerlocspace$]
\label{res:BV_K_lsc}
If $(u_k)_{k\in\N}\subset\BVkerlocspace(\Rdim)$ is such that $u_k\to u$ in $\Llocspace{1}(\Rdim)$ as $k\to+\infty$ with
$\liminf\limits_{k\to+\infty}
\kerTV{u_k}<+\infty$,
then $u\in\BVkerlocspace(\Rdim)$ with
\begin{equation*}
\kerTV{u}
\le
\liminf_{k\to+\infty}
\kerTV{u_k}.
\end{equation*}
\end{lemma}

\begin{lemma}[Lower semicontinuity for $\Pker$]
\label{res:P_K_lsc}
Let $A\in\measurablesets$. 
If $(U_k)_{k\in\N}\subset\measurablesets$ is such that $\chi_{U_k}\to\chi_U$ in $\Llocspace{1}(\Rdim)$ as $k\to+\infty$ with
$\liminf\limits_{k\to+\infty}
\Pker{(U_k;A)}<+\infty$,
then \begin{equation*}
\Pker{(U;A)}
\le
\liminf_{k\to+\infty}
\Pker{(U_k;A)}.
\end{equation*}
\end{lemma}

\subsection{Compactness}
\label{subsec:compactness}

We now prove the compactness of the embedding $\BVkerspace(\Rdim)\subset\Llocspace{1}(\Rdim)$.
Actually, we prove the more general \cref{res:compactness_p} below, which extends~\cite{JW20}*{Th.~1.1} and, for $p=1$, positively answers a question left open in~\cite{CN18} (see~\cite{CdP20} for similar results in Orlicz spaces).
After we completed our paper and posted it on arXiv, G.~F.~Foghem Gounoue kindly indicated his Ph.D.\ thesis~\cite{F-PhD} to us (of which we were unaware) where he proved \cref{res:compactness_p} below under the additional~\eqref{H:Symmetric}, see \cref{rem:ass-compact_p}.

Let $\kernel\colon\Rdim\to[0,+\infty]$ be a kernel and let $p\in[1,+\infty)$.
Given $u\in\Llocspace{1}(\Rdim)$, we let
\begin{equation*}
[u]_{W^{\kernel,p}(\Rdim)}
=
\left(\frac12
\int_{\Rdim}\int_{\Rdim}
|u(x)-u(y)|^p\,\kernel(x-y)\,\de x\,\de y
\right)^{1/p}
\end{equation*}
and we define 
\begin{equation*}
W^{\kernel,p}_{\loc}(\Rdim)
=
\set*{u\in\Llocspace{1}(\Rdim)
:
[u]_{W^{\kernel,p}(\Rdim)}<+\infty
}
\end{equation*}
and $W^{\kernel,p}(\Rdim)=W^{\kernel,p}_{\loc}(\Rdim)\cap\Lspace{p}(\Rdim)$.
The space $W^{\kernel,p}(\Rdim)$ is Banach with respect to
\begin{equation*}
\|u\|_{W^{\kernel,p}(\Rdim)}
=
\|u\|_{\Lspace{p}(\Rdim)}
+
[u]_{W^{\kernel,p}(\Rdim)},
\quad
u\in L^p(\Rdim).
\end{equation*} 
If $p=1$, then $W^{\kernel,p}(\Rdim)=\BVkerspace(\Rdim)$.
Since $\kernel\not\equiv0$, if $[u]_{W^{\kernel,p}(\Rdim)}=0$ for $u\in\Llocspace{1}(\Rdim)$, then $u=c$ a.e.\ for some $c\in\R$, with $c=0$ if $u\in\Lspace{p}(\Rdim)$.
Note that $W^{\kernel,p}(\Rdim)=\Lspace{p}(\Rdim)$ whenever $K\in\Lspace{1}(\Rdim)$, since
\begin{equation*}
[u]_{W^{\kernel,p}(\Rdim)}
\le 
2^{\frac{p-1}p}\|\kernel\|_{\Lspace{1}(\Rdim)}^{1/p}\,\|u\|_{\Lspace{p}(\Rdim)}
\end{equation*}
by definition.
In addition, the condition
\begin{equation}
\label{H:Not_too_singular_p}
\tag{Nts$_p$}
\int_{\Rdim}(|x|^p\wedge1)\,\kernel(x)\,\de x<+\infty.\end{equation}
yields the continuous embedding $W^{1,p}(\Rdim)\subset W^{\kernel,p}(\Rdim)$, with
\begin{equation}
\label{eq:valerio_p_embed}
[u]_{W^{\kernel,p}(\Rdim)}
\le 
2^{-1/p}\,\max\set*{\|\nabla u\|_{\Lspace{p}(\Rdim;\Rdim)},2\|u\|_{\Lspace{p}(\Rdim)}}
\left(\,\int_{\Rdim}(|x|^p\wedge1)\,\kernel(x)\,\de x\right)^{1/p}
\end{equation}
for all $u\in W^{1,p}(\Rdim)$.
Finally, if~\eqref{H:Far_from_zero_int} holds, then 
$W^{\kernel,p}(\Rdim)=W^{\kernel\chi_{B_r},p}(\Rdim)$
for all $r>0$,
with equivalence of the norms.

\begin{theorem}[Sequential compactness in $W^{\kernel,p}$]
\label{res:compactness_p}
Let \eqref{H:Far_from_zero_int} and~\eqref{H:Not_integrable} be in force.
If $(u_k)_{k\in\N}\subset W^{\kernel,p}(\Rdim)$ satisfies
$\sup\limits_{k\in\N}\|u_k\|_{W^{\kernel,p}(\Rdim)}<+\infty$,
then there exists a subsequence $(u_{k_j})_{j\in\N}\subset W^{\kernel,p}(\Rdim)$ and $u\in W^{\kernel,p}(\Rdim)$ such that 
$u_{k_j}\to u$ in $\Llocspace{p}(\Rdim)$ as $j\to+\infty$.
\end{theorem}

\begin{remark}[On the assumptions of~\cite{JW20}*{Th.~1.1 and Lem.~2.2} and~\cite{F-PhD}*{Th.~3.81}]
\label{rem:ass-compact_p}
In~\cite{JW20}, the authors assume~\eqref{H:Not_integrable} and~\eqref{H:Not_too_singular_p} (and also~\eqref{H:Symmetric} without loss of generality).
However, one can relax~\eqref{H:Not_too_singular_p} to~\eqref{H:Far_from_zero_int}, as we do here in \cref{res:compactness_p} (as well as in \cref{res:convolution_distance} below).
Indeed, \eqref{H:Not_too_singular_p} ensures that $W^{\kernel,p}(\Rdim)\not=\set*{0}$, thanks to~\eqref{eq:valerio_p_embed} (if $W^{\kernel,p}(\Rdim)=\set*{0}$, then \cref{res:compactness_p} is obviously true).
We also mention that~\cite{F-PhD}*{Th.~3.81} assumes~\eqref{H:Symmetric}. However, as we do here, this  is not needed. 
\end{remark}

In \cref{res:compactness_p}, we do not assume~\eqref{H:Radial}.
In particular, \cref{res:compactness_p} provides an alternative proof of the existence of  isoperimetric sets among bounded sets for the anisotropic fractional perimeter  considered in~\cites{L14,K21} (in particular, see~\cite{L14}*{Eq.~(10)}).
Similarly, we can state the following analogous result, whose plain proof is omitted.

\begin{corollary}[Conditional isoperimetric inquality]
Let~\eqref{H:Far_from_zero_int} and~\eqref{H:Not_integrable} be in force.
Let $v\in(0,+\infty)$ and $\Omega\in\measurablesets$ with $|\Omega|<+\infty$ be fixed. 
If 
\begin{equation}
\label{eq:isop_conditional}
\inf\big\{\Pker(E) : E\in\measurablesets,\ |E|=v,\ E\subset\Omega\big\}<+\infty,
\end{equation}
then the conditional isoperimetric problem~\eqref{eq:isop_conditional} admits minimizers.
\end{corollary}

Given $E\in\measurablesets$, we let
$\mathcal R_E\colon 
\Llocspace{1}(\Rdim)
\to 
\Llocspace{1}(\Rdim)$ be given by $\mathcal R_E(u)=u\chi_E$ for $u\in\Llocspace{1}(\Rdim)$.
Note that $\mathcal R_E$ is continuous from $\Lspace{p}(\Rdim)$ to itself whenever $p\in[1,+\infty]$.

\begin{definition}[Locally compact operator on $\Lspace{p}$]
Let $p\in[1,+\infty)$.
We say that $T\colon\Lspace{p}(\Rdim)\to\Lspace{p}(\Rdim)$ is \emph{locally compact} if $\mathcal R_E\circ T\colon\Lspace{p}(\Rdim)\to\Lspace{p}(\Rdim)$ is compact whenever $E\subset\Rdim$ is a compact set.
\end{definition}

Since any strong limit of compact operators is also compact (see~\cite{B11}*{Th.~6.1} for example), any strong limit of locally compact operators is locally compact too.  

The following result generalizes~\cite{JW20}*{Lem.~2.1} to all exponents $p\in[1,+\infty)$.
We also refer the reader to~\cite{F-PhD}*{Th.~3.9} for an even more general result. 
Its simple proof follows the one of~\cite{JW20}*{Lem.~2.1} and is thus omitted.

\begin{lemma}[Convolution is locally compact]
\label{res:convolution_compact}
Let $p\in[1,+\infty)$.
If $\eta\in\Lspace{1}(\Rdim)$, then $T_\eta\colon\Lspace{p}(\Rdim)\to\Lspace{p}(\Rdim)$, $T_\eta(u)=u*\eta$ for $u\in\Lspace{p}(\Rdim)$, is locally compact.
\end{lemma}

The following results generalizes~\cite{JW20}*{Lem.~2.2} for all $p\in[1,+\infty)$ (also see the proof of~\cite{F-PhD}*{Th.~3.81}).
Its simple proof follows the one of~\cite{JW20}*{Lem.~2.2} and is thus omitted.

\begin{lemma}[$L^p$ distance to convolution in $W^{\kernel,p}$]
\label{res:convolution_distance}
Let~\eqref{H:Far_from_zero_int} be in force.
Let $\delta>0$ be such that $\kernel_\delta=\kernel\chi_{\Rdim\setminus B_\delta}\in\Lspace{1}(\Rdim)\setminus\set*{0}$. 
If 
$\eta_\delta
=
\frac{\kernel_\delta}{\|\kernel_\delta\|_{\Lspace{1}(\Rdim)}}$,
then 
\begin{equation*}
\|u-u*\eta_\delta\|_{\Lspace{p}(\Rdim)}
\le
2^{1/p}\,
\|\kernel_\delta\|_{\Lspace{1}(\Rdim)}^{-1/p}
\,
[u]_{W^{\kernel,p}(\Rdim)}
\quad
\text{for}\ 
u\in W^{\kernel,p}(\Rdim).
\end{equation*}
\end{lemma}

\begin{proof}[Proof of \cref{res:compactness_p}]
The proof follows the same strategy of~\cite{JW20}, so we only sketch it.
Let $S\subset W^{\kernel,p}(\Rdim)$ be a bounded and $C=\sup\limits_{u\in S}\|u\|_{W^{\kernel,p}(\Rdim)}$. 
Given $E\subset\Rdim$ compact, we have to show that $\mathcal R_E(S)\subset\Lspace{p}(\Rdim)$ is relatively compact. 
Given $\eps>0$, by~\eqref{H:Far_from_zero_int} and~\eqref{H:Not_integrable}  we can find $\delta>0$ such that 
$
2^{1/p}\,
\|\kernel_\delta\|_{\Lspace{1}(\Rdim)}^{-1/p}
\le \frac \eps C,
$
so that 
\begin{equation*}
\begin{split}
\|\mathcal R_E(u)-\mathcal R_E T_{\eta_\delta}(u)\|_{\Lspace{p}(\Rdim)}
\le 
\|u-T_{\eta_\delta}(u)\|_{\Lspace{p}(\Rdim)}
\le 
2^{1/p}\,
\|\kernel_\delta\|_{\Lspace{1}(\Rdim)}^{-1/p}
\,
[u]_{W^{\kernel,p}(\Rdim)}
\le\eps
\end{split}
\end{equation*}
for all $u\in S$ by \cref{res:convolution_distance}, where $T_{\eta_\delta}(u)=u*\eta_\delta$.
Therefore, $\mathcal R_E(S)$ is contained in an $\eps$-neighborhood of the set $\mathcal R_E T_{\eta_\delta}(S)$, which is relatively compact according to \cref{res:convolution_compact}. 
Hence $\mathcal R_E(S)$ is totally bounded in $\Lspace{p}(E)$ and thus relatively compact.
\end{proof}

\subsection{Lusin-type estimate}
\label{subsec:lusin}

With the same notation of  \cref{subsec:compactness}, the following result provides a\emph{ Lusin-type estimate} for $W^{\kernel,p}$ functions, generalizing~\cite{BN20}*{Th.~1.11 and Prop.~1.13}.
Here and in the following, we let
\begin{equation}
\label{eq:def_phi_kappa}
\phi_\kernel(\eps,R)
=
\int_{B_R\setminus B_\eps}\kernel(x)\,\de x
\end{equation}
for $R>\eps>0$.
If $R=+\infty$, then we simply let
\begin{equation*}
\phi_\kernel(\eps)
=
\phi_\kernel(\eps,+\infty)
=
\int_{\Rdim\setminus B_\eps}\kernel(x)\,\de x
\end{equation*}
for $\eps>0$.
Note that $0\le \phi_\kernel(\eps,R)<+\infty$ for all $R>\eps>0$ as soon as~\eqref{H:Far_from_zero_int} is in force.  

\begin{theorem}[Lusin-type estimate in $W^{\kernel,p}$]
\label{res:lusin}
Let 
\eqref{H:Far_from_zero_int},
\eqref{H:Decreasing_q}
and
\eqref{H:Doubling}
be in force.
Let $p\in[1,+\infty)$.
If $u\in W^{K,p}(\Rdim)$, then 
\begin{equation*}
|u(x)-u(y)|^p
\,\phi_\kernel(2|x-y|,D)
\le 
C\big(\kerNablap{u}(x)+\kerNablap{u}(y)\big)
\end{equation*}
for a.e.\ $x,y\in\Rdim$ with $|x-y|\le\frac D2$ such that $\kerNablap{u}(x),\kerNablap{u}(y)<+\infty$, where
\begin{equation*}
\kerNablap{u}(x)=\frac{1}{2}\int_{\Rdim}|u(x)-u(y)|^p\,\kernel(x-y)\,\de y
\in[0,+\infty]
\end{equation*}
for $x\in\Rdim$.
As a consequence, we have
\begin{equation*}
\|u(\cdot+h)-u\|_{\Lspace{p}(\Rdim)}
\,
\phi_\kernel(2|h|,D)
\le
C[u]_{W^{\kernel,p}(\Rdim)}
\end{equation*}
for all $h\in\Rdim$ with $|h|\le\frac D2$.
The constant $C>0$ depends on $\dimension$, $p$ and $\kernel$ only.
\end{theorem}

\begin{proof}
The proof is similar to that of~\cite{BN20}*{Th.~1.11}, so we only sketch it.
Let $R>0$ to be chosen later.
We fix $x,y\in\Rdim$ with $2|x-y|\le R$ such that all the quantities appearing below are well defined and such that~\cite{BN20}*{Lem.~1.12} is applicable.
By~\eqref{eq:def_phi_kappa} and~\cite{BN20}*{Lem.~1.12}, we can estimate
\begin{align*}
|u(x)-u(y)|^p\,\phi_\kernel(2|x-y|,R)
&\le 
c_{\dimension,p}
\int_{B_R\setminus B_{2|x-y|}} \kernel(\zeta)\,\aint_{B_{3|\zeta|}\setminus B_{|\zeta|}}|u(x+h)-u(x)|^p\,\de h\,\de \zeta
\\
&\quad+
c_{\dimension,p}\int_{B_R\setminus B_{2|x-y|}} \kernel(\zeta)\,\aint_{B_{3|\zeta|}\setminus B_{|\zeta|}}|u(y+h)-u(y)|^p\,\de h\,\de \zeta,
\end{align*}
where $c_{\dimension,p}>0$ is a constant depending on $\dimension$ and $p$ only.
We thus can easily bound
\begin{align*}
\int_{B_R\setminus B_{2|x-y|}} \kernel(\zeta)\,
&
\aint_{B_{3|\zeta|}\setminus B_{|\zeta|}}|u(x+h)-u(x)|^p\,\de h\,\de \zeta
\\
&\le
C\int_{\Rdim}|u(x+h)-u(x)|^p\int_{\Rdim} \kernel(\zeta)\,\chi_{B_R\setminus B_{2|x-y|}}(\zeta)\chi_{B_{3|\zeta|}\setminus B_{|\zeta|}}(h)\,\frac{\de \zeta}{|\zeta|^\dimension}\,\de h,
\end{align*}
where $C>0$ is a dimensional constant. 
Now 
$
\kernel(\zeta)
\le 
C\kernel(2\zeta)
\le
C^2\kernel(4\zeta)
$
provided that $2|\zeta|\le2D$, where $C>0$ is the constant appearing in~\eqref{H:Doubling}. 
We can then estimate
$
\kernel(4\zeta)
\le 
\kernel(3\zeta)
\le 
\kernel(h)
$
for all $h\in B_{3|\zeta|}\setminus B_{|\zeta|}$
by~\eqref{H:Decreasing_q}.
Therefore, we have 
$
\kernel(\zeta)\le C\kernel(h)
$
for all $\zeta\in B_R\setminus B_{2|x-y|}$ and $h\in B_{3|\zeta|}\setminus B_{|\zeta|}$
provided that we choose $R=D$, where $C>0$ depends on $\kernel$ only.
Hence, we get that
\begin{align*}
\int_{\Rdim}|u(x+h)&-u(x)|^p\int_{\Rdim} \kernel(\zeta)\,\chi_{B_{D}\setminus B_{2|x-y|}}(\zeta)\chi_{B_{3|\zeta|}\setminus B_{|\zeta|}}(h)\,\frac{\de \zeta}{|\zeta|^\dimension}\,\de h
\\
&\le
C\int_{\Rdim}|u(x+h)-u(x)|^p\,\kernel(h)\int_{\set*{\zeta\in\Rdim:\max\set*{2|x-y|,\frac{|h|}3}\leq|\zeta|\leq\min\set*{|h|,D}}}\frac{\de \zeta}{|\zeta|^\dimension}\,\de h
\\
&\le
C\int_{\Rdim}|u(x+h)-u(x)|^p\,K(h)\int_{\frac{|h|}3}^{|h|}\frac{\de r}r\,\de h
\\
&\le
C\kerNablap{u}(x),
\end{align*}
where $C>0$ depends on $\dimension$ and $\kernel$ only (and possibly varies from line to line).
The conclusion hence follows by combining the above estimates and swapping~$x$ and $y$.
\end{proof}

\begin{remark}[$\Lspace{p}$ distance to convolution]
\label{res:smoothing_oscillation}
Under the assumptions of \cref{res:lusin}, there exists  $C>0$, depending on $\dimension$, $p$ and $\kernel$ only, such that
\begin{equation} \label{eq:smoothing_oscillation}
\|\rho_{\eps}*u - u\|_{\Lspace{p}(\Rdim)}
\le
C\,\ell_\kernel(\eps,D)\,
[u]_{W^{\kernel,p}(\Rdim)}
\end{equation}
for all $u \in W^{K,p}(\Rdim)$ and $\eps\in\left(0,\frac D2\right]$.
Here $(\rho_\eps)_{\eps>0}$ is a family of convolution kernels,
$
\rho_\eps
=
\eps^{-\dimension}\rho\left(\frac \cdot\eps\right)$,
where $\rho\in C^\infty_c(\R^n)$ is such that 
$\supp\rho\subset B_1$,
$\rho\ge0$,
$\int_{\Rdim}\rho\,\de x=1$, and
\begin{equation*}
\ell_\kernel(\eps,R)
=
\int_{B_1}
\frac{\rho(y)}{\phi_\kernel(2\eps|y|,R)}
\,\de y
\quad
\text{for all $R>\eps>0$},
\end{equation*}
where $\phi_K$ is as in~\eqref{eq:def_phi_kappa}.
Note that, if~\eqref{H:Far_from_zero_int} holds, then $\ell_K$ is well-defined with
$
0<\ell_\kernel(\eps,R)
\le
1/\phi_K(2\eps,R)
\le 
+\infty
$
for all $R>\eps>0$.
Moreover, if~\eqref{H:Far_from_zero_int} and~\eqref{H:Not_integrable} are in force, then 
$
\lim_{\eps\to0^+}\ell_K(\eps,R)=0
$
for each $R>0$ by the Monotone Convergence Theorem. 
Note that~\eqref{eq:smoothing_oscillation} implies \cref{res:compactness_p}, although under stronger assumptions.
\end{remark}

\subsection{Isoperimetric inequality}

\label{subsec:isoperimetry}

For a more detailed presentation of the following notation, see~\cite{LL01}*{Ch.~3} and~\cite{DNP21}*{App.~A}.
We set 
\begin{equation}
\label{eq:ballification}
B^v=B_{r_v},
\quad
r_v=\left(\frac{v}{|B_1|}\right)^{1/n},
\quad
\text{for}\
v>0.
\end{equation}
Given $A\in\measurablesets$ with $|A|<+\infty$, we let $A^\bigstar=B^{|A|}$ as in~\eqref{eq:ballification}.
Consequently, we set $\chi_A^\bigstar=\chi_{A^\bigstar}$ and thus, whenever $f\colon\R^n\to[-\infty,+\infty]$ is a measurable function such that $|\set*{|f|>t}|<+\infty$ for all $t>0$, i.e., \emph{$f$ vanishes at infinity}, we let
\begin{equation}
\label{eq:def_rearrangement}
f^\bigstar(x)=\int_0^{+\infty}\chi_{\set*{|f|>t}}^\bigstar(x)\,\de t,
\quad
x\in\Rdim,
\end{equation}
be the \emph{symmetric-decreasing rearrangement} of~$f$.
We recall that 
$
\set*{f^\bigstar>t}=\set*{|f|>t}^\bigstar
$
for all $t\in\R$.
Note that $f^\bigstar=f$ whenever $f(x)=\phi(|x|)$ for $x\in\Rdim$, where $\phi\colon[0,+\infty)\to[0,+\infty]$ is decreasing.
The \emph{Riesz rearrangement inequality} hence states that
\begin{equation}
\label{eq:Riesz_rearrang}
\int_{\Rdim}
\int_{\Rdim}
f(x)\,g(x-y)\,h(y)\,\de x\,\de y
\le 
\int_{\Rdim}
\int_{\Rdim}
f^\bigstar(x)\,g^\bigstar(x-y)\,h^\bigstar(y)\,\de x\,\de y
\end{equation}
whenever $f,g,h\colon\Rdim\to[0,+\infty]$ vanish at infinity.

\begin{theorem}[Isoperimetric inequality]
\label{res:isoperimetric}
Let~\eqref{H:Radial} and~\eqref{H:Decreasing_q} be in force.
If $E\in\measurablesets$ is such that $\chi_E\in\BVkerspace(\Rdim)$,
then
\begin{equation}
\label{eq:isoperimetric}
\Pker(E)\ge\Pker(B^{|E|}).
\end{equation}
In addition, if \eqref{H:Radial_strict} is in force, then~\eqref{eq:isoperimetric} holds as an equality if and only if $E$ is a translated of~$B^{|E|}$ (up to negligible sets). 
\end{theorem}

The above \cref{res:isoperimetric} is a particular case of~\cite{CN18}*{Prop.~3.1}.
However, in~\cite{CN18}*{Prop.~3.1}, the authors state that equality in~\eqref{eq:isoperimetric} occurs if $\kernel$ satisfies~\eqref{H:Radial} and~\eqref{H:Decreasing_q}.
In fact, in the first step of the proof of~\cite{CN18}*{Prop.~3.1} (under the additional assumption $\kernel\in\Lspace{1}(\Rdim)$), they assert that these two assumptions are enough to characterize the cases of equality in~\eqref{eq:Riesz_rearrang}. Unfortunately, this is not correct, as some further assumptions on $\kernel$ are needed, see~\cite{B96} for a more detailed discussion. 
Moreover, in the proof of~\cite{CN18}*{Prop.~3.1}, the characterization of the equality is not explicitly treated in the general case. 
For these reasons, we provide a proof of \cref{res:isoperimetric} where we characterize the equality case in~\eqref{eq:isoperimetric} under the additional~\eqref{H:Radial_strict}, following the strategy of~\cite{DNP21}*{App.~A}.

\begin{proof}[Proof of \cref{res:isoperimetric}]
Assume $|E|>0$ and note that $\|\kernel\|_{\Lspace{\infty}(\Rdim)}\in(0,+\infty]$, because $\kernel\not\equiv0$.
Thanks to~\eqref{H:Radial} and~\eqref{H:Decreasing_q}, for every $t\geq 0$ there exists $R(t)\in[0,+\infty]$ such that
$\set*{\kernel>t}=B_{R(t)}$,
with $R(t)\in(0,+\infty]$ for $t\in(0,\|\kernel\|_{\Lspace{\infty}(\Rdim)})$ and $R(t)=0$ for $t\geq\|\kernel\|_{\Lspace{\infty}(\Rdim)}$.
By Tonelli's Theorem, we can write
\begin{align*}
	+\infty>\Pker(E)&=\int_{\Rdim}\int_{\Rdim}\chi_{E}(x)\,\chi_{\comp{E}}(y)\,\kernel(x-y)\,\de x\,\de y\\
	        &=\int_{0}^{+\infty}\int_{\Rdim}\int_{\Rdim}\chi_{E}(x)\,\chi_{\comp{E}}(y)\,\chi_{\set*{\kernel>t}}(x-y)\,\de x\,\de y\,\de t,
\end{align*}
so that
\begin{equation}\label{IntegrandaFinita}
	\int_{\Rdim}\int_{\Rdim}\chi_{E}(x)\chi_{\comp{E}}(y)\chi_{\set*{\kernel>t}}(x-y)\,\de x\,\de y<+\infty
\end{equation}
for $\lebone$-a.e.\ $t>0$.
Now we fix $t\in(0,\|\kernel\|_{\Lspace{\infty}(\Rdim)})$ such that \eqref{IntegrandaFinita} holds and we claim that $R(t)<+\infty$. Indeed, if $R(t)=+\infty$ by contradiction, then~\eqref{IntegrandaFinita} leads to
\begin{equation*}
	+\infty>\int_{\Rdim}\int_{\Rdim}\chi_{E}(x)\,\chi_{\comp{E}}(y)\,\de x\,\de y=|E||\comp{E}|=+\infty,
\end{equation*}
which is impossible.
Therefore, since $t\mapsto\set*{\kernel>t}$ is decreasing with respect to inclusion, $\chi_{\set*{\kernel>t}}\in\Lspace{1}(\Rdim)$ for every $t>0$. 
Now, by~\eqref{eq:L1_ker_perimeter} (applied to $\chi_{\set*{\kernel>t}}$) we rewrite~\eqref{IntegrandaFinita} as
\begin{align*}
\int_{\Rdim}\int_{\Rdim}&\chi_{E}(x)\,\chi_{\comp{E}}(y)\,\chi_{\set*{\kernel>t}}(x-y)\,\de x\,\de y\\
		&=|E||B_{R(t)}|-\int_{\Rdim}\int_{\Rdim}\chi_{E}(x)\,\chi_{E}(y)\chi_{B_{R(t)}}(x-y)\,\de x\,\de y
\end{align*}
for every $t>0$.
In conclusion, we get that 
\begin{align}\label{eq:formulaPerPK}
	\Pker(F)=\int_{0}^{\|\kernel\|_{\Lspace{\infty}(\Rdim)}}\left(|F||B_{R(t)}|-\int_{\Rdim}\int_{\Rdim}\chi_{F}(x)\,\chi_{F}(y)\,\chi_{B_{R(t)}}(x-y)\,\de x\,\de y\right)\de t
\end{align}
for $F\in\measurablesets$ such that $\Pker(F)<+\infty$ and $|F|<+\infty$.
Inequality~\eqref{eq:isoperimetric} thus follows by~\eqref{eq:formulaPerPK} and~\eqref{eq:Riesz_rearrang}.
Finally, if \eqref{H:Radial_strict} holds, then $R(t)\searrow 0^+$ as $t\nearrow\|\kernel\|_{\Lspace{\infty}(\Rdim)}$. 
So 
$
R(t)\in(0,2r_{|E|})
$
for $t\in\R$ close to $\|\kernel\|_{\Lspace{\infty}(\Rdim)}$. 
By \cite{B96}*{Th.~1} and~\eqref{eq:formulaPerPK}, we get $\Pker(E)=\Pker(B^{|E|})$ if and only if $E$ is equivalent to a ball.
\end{proof}

\begin{remark}[A question left open in~\cite{BN20}]
\cref{res:isoperimetric} affirmatively answers a question left open in~\cite{BN20}*{p.~842} concerning the isoperimetric problem for the non-local perimeter associated to 
$\kernel_\gamma(x)=|x|^{-\dimension}\,|\log|x||^{\gamma-1}\,\chi_{B_{1/3}}(x)$,
$x\in\Rdim\setminus\set{0}$,
whenever $\gamma\ge0$.
Indeed, $\kernel_\gamma$ is radial and satisfies~\eqref{H:Decreasing_q} for all $\gamma\ge0$. 
\end{remark}

The following result generalizes~\cite{AL89}*{Th.~9.2} and~\cite{FS08}*{Th.~A.1}.
We omit its plain proof.

\begin{theorem}[Rearrangement inequality]
\label{res:rearrangement}
Let~\eqref{H:Radial} and~\eqref{H:Decreasing_q} be in force.
If $u\in\BVkerspace(\Rdim)$, then also $u^\bigstar\in\BVkerspace(\Rdim)$ with 
\begin{equation}
\label{eq:rearrangement_isop}
\kerTV{u}\ge\kerTV{u^\bigstar}.
\end{equation}
In addition, if also \eqref{H:Radial_strict} and~\eqref{H:Positive} are in force, then~\eqref{eq:rearrangement_isop} holds as an equality if and only if $u$ is proportional to a function $v$ such that $v(x)\ge0$ for $\lebdim$-a.e.\ $x\in\Rdim$ and  $\set*{v>t}$ is a ball (up to negligible sets) for $\lebone$-a.e.\ $t>0$. 
\end{theorem}

\subsection{The isoperimetric function}

\label{subsec:isop_function}

We let $\beta_\kernel\colon[0,+\infty)\to[0,+\infty]$,
\begin{equation}
\label{eq:def_beta_kernel}
\beta_\kernel(v)
=
\mathcal{P}_{\kernel}(B^v)
\quad
\text{for}\ v>0,
\end{equation}
be the \emph{isoperimetric function} (recall the notation in~\eqref{eq:ballification}).
The following result is a simple consequence of~\cite{CN18}*{Lem.~3.2} and we thus omit its proof.

\begin{lemma}[Behavior of $\beta_\kernel$]
\label{res:behavior_beta}
Let \eqref{H:Radial} and \eqref{H:Decreasing_q} be in force.
It holds
\begin{equation*}
\lim_{v\to0^+}
\frac{\beta_\kernel(v)}{v}
=
\begin{cases}
\|\kernel\|_{\Lspace{1}(\Rdim)}
&
\text{if}\ \kernel\in\Lspace{1}(\Rdim),
\\[3mm]
+\infty
&
\text{otherwise}.
\end{cases}
\end{equation*}
\end{lemma}

The following dichotomy result is a consequence of \cref{res:isoperimetric}.
Its proof follows the same strategy of that of~\cite{CN18}*{Prop.~3.3 and Cor.~3.4} and is thus left to the reader (to this aim, recall that \cref{res:dec_q_comparison} ensures the validity of~\eqref{H:Inf}, which is needed  in~\cite{CN18}).

\begin{corollary}[Dichotomy]
\label{res:dichotomy}
Let 
\eqref{H:Radial} 
and
\eqref{H:Decreasing_q} be in force.
If $E\in\measurablesets$ is such that $\Pker(E)<+\infty$, then either $|E|<+\infty$ or $|E^c|<+\infty$, and
\begin{equation*}
\Pker(E)
\ge
\min\big\{\beta_\kernel(|E|),\beta_\kernel(|E^c|)\big\}.
\end{equation*}
\end{corollary}

With a slight abuse of notation, we let
\begin{equation*}
L^{\beta_\kernel(\cdot),1}(\Rdim)
=
\set*{u\colon\Rdim\to[-\infty,+\infty]\ \text{measurable} : \|u\|_{L^{\beta_\kernel(\cdot),1}(\Rdim)}
<+\infty},
\end{equation*}
where 
\begin{equation*}
\|u\|_{L^{\beta_\kernel(\cdot),1}(\Rdim)}
=
\int_0^{+\infty}\beta_\kernel\big(|\set*{|u|>t}|\big)\,\de t.
\end{equation*}
Since
$
\|u\|_{L^{\beta_\kernel(\cdot),1}(\Rdim)}=\kerTV{u^\bigstar}
$
by \cref{res:coarea}, from  \cref{res:rearrangement} we can infer the following Sobolev-type embedding, encoding~\cite{DPV12}*{Th.~6.5}, \cite{FS08}*{Th.~4.1} and~\cite{BN20}*{Th.~1.5} when $p=1$ (also compare with~\cite{F21}*{Th.~1.1}). 

\begin{corollary}[Sobolev isoperimetric embedding]
\label{res:sobolev}
Let
\eqref{H:Radial} and
\eqref{H:Decreasing_q}  
be in force.
The embedding 
$
\BVkerspace(\Rdim)\subset L^{\beta_\kernel(\cdot),1}(\Rdim)
$
is continuous, with 
\begin{equation}
\label{eq:sobolev_ineq}
\|u\|_{L^{\beta_\kernel(\cdot),1}(\Rdim)}
\le 
\kerTV{u}
\end{equation}
for $u\in\BVkerspace(\Rdim)$.
In addition, if also \eqref{H:Radial_strict} and~\eqref{H:Positive} are in force, then~\eqref{eq:sobolev_ineq} is an equality if and only if
$u$ is proportional to a function $v$ such that $v(x)\ge0$ for $\lebdim$-a.e.\ $x\in\Rdim$ and  $\set*{v>t}$ is a ball (up to negligible sets) for $\lebone$-a.e.\ $t>0$.  
\end{corollary}

\subsection{Sobolev embeddings}

\label{subsec:sobolev}

We now refine \cref{res:isoperimetric} and \cref{res:sobolev}.

\begin{lemma}[Monotonicity]
\label{res:monotonicity_isop_ratio}
Let 
\eqref{H:Decreasing_q} with $q<\dimension+1$
be in force.
If $E\in\measurablesets$ with $|E|\in(0,+\infty)$, then
\begin{equation*}
0<r\le R<+\infty
\implies
\frac{\Pker(rE)}{|rE|^{2-\frac{q}\dimension}}
\ge
\frac{\Pker(RE)}{|RE|^{2-\frac{q}\dimension}}
\end{equation*}
In particular, 
$v\mapsto\beta_\kernel(v)\,v^{\frac q\dimension-2}$
is decreasing for $v\in(0,+\infty)$.
If  \eqref{H:Decreasing_dim_strinct} holds, then 
\begin{equation*}
0<r<R<+\infty
\implies
\frac{\Pker(rE)}{|rE|}
>
\frac{\Pker(RE)}{|RE|}
\end{equation*}
and thus $v\mapsto\beta_\kernel(v)\,v^{-1}$ is strictly decreasing for $v\in(0,+\infty)$.
\end{lemma}

\begin{proof}
Since
\begin{equation*}
\Pker(F)
=
\frac12\int_F\int_{F^c}\kernel(x-y)\,\de x\,\de y
+
\frac12\int_{F^c}\int_{F}\kernel(x-y)\,\de x\,\de y
\end{equation*}
whenever $F\in\measurablesets$, in virtue of~\eqref{H:Decreasing_q}, we can estimate
\begin{align*}
\int_{RE}\int_{(RE)^c}
\kernel(x-y)
\,\de x\,\de y
&=
R^{2\dimension}\int_{E}\int_{E^c}
\kernel(R(\xi-\eta))
\,\de\xi\,\de\eta
\\
&=
R^{2\dimension}\int_{E}\int_{E^c}
\frac{\kernel(R(\xi-\eta))
\,(R|\xi-\eta|)^q}{(R|\xi-\eta|)^q}\,\de\xi\,\de\eta
\\
&\le
R^{2\dimension}\int_{E}\int_{E^c}
\frac{\kernel(r(\xi-\eta))
\,(r|\xi-\eta|)^q}{(R|\xi-\eta|)^q}\,\de\xi\,\de\eta
\\
&=
R^{2\dimension-q}\,r^q
\int_{E}\int_{E^c}
\kernel(r(\xi-\eta))
\,\de\xi\,\de\eta
\\
&=
\frac{R^{2\dimension-q}}{r^{2\dimension-q}}
\int_{rE}\int_{(rE)^c}
\kernel(x-y)
\,\de x\,\de y
\end{align*}
whenever $0<r\le R<+\infty$, with the unique inequality strict for $r<R$ provided that \eqref{H:Decreasing_dim_strinct} holds. 
A similar estimate holds for the integral relative to $E^c\times E$.
The conclusion hence follows by rearranging and by the definition of $\beta_\kernel$ in~\eqref{eq:def_beta_kernel}.
\end{proof}

 \cref{res:monotonicity_isop_ratio} implies the following isoperimetric-type inequality for small volumes.

\begin{proposition}[Isoperimetric inequality  for small volume]
\label{res:isoperimetric_small_mass}
Let
\eqref{H:Radial}
and 
\eqref{H:Decreasing_q} with $q<\dimension+1$ 
be in force.
If $\chi_E\in\BVkerspace(\R^n)$ with $|E|\le v$ for some $v\in(0,+\infty)$, then
\begin{equation*}
\Pker(E)
\ge 
\frac{\beta_\kernel(v)}{v^{2-\frac q\dimension}}\,|E|^{2-\frac q\dimension}.
\end{equation*}
\end{proposition}

Letting
$\Lspace{p,1}(\Rdim)$ be the  \emph{Lorentz $(p,1)$-space} for $p\in(0,+\infty)$ (see~\cite{G14-C}*{Sec.~1.4} for an account), from \cref{res:isoperimetric_small_mass} we readily get the following result. 

\begin{corollary}[Sobolev embedding for finite-measure support]
\label{res:sobolev_finite_supp}
Let 
\eqref{H:Radial}
and 
\eqref{H:Decreasing_q} with $q<\dimension+1$ 
be in force.
If $u\in\BVkerspace(\Rdim)$ is such that $|\supp u|<+\infty$, then 
\begin{equation*}
\kerTV{u}
\ge
\left(2-\tfrac q\dimension\right) 
\frac{\beta_\kernel(|\supp u|)}{|\supp u|^{2-\frac q\dimension}}
\,
\|u\|_{\Lspace{\frac \dimension{2\dimension-q},1}(\Rdim)}.
\end{equation*}
\end{corollary}

\cref{res:isoperimetric_small_mass} pairs with the following result, whose plain proof is omitted.

\begin{proposition}[Isoperimetric inequality for large volumes]\label{res:isoperimetric_large_volume}
Let 
\eqref{H:Radial},
\eqref{H:Decreasing_q} 
with $q<\dimension+1$ 
and
\eqref{H:Doubling} 
with $D=+\infty$
be in force.
Let the doubling constant of $\kernel$ be such that $C>2^n$. 
Given $V\in(0,+\infty)$, there exists $C_{V,\kernel,\dimension}>0$
such that 
\begin{equation}
\label{eq:isoperimeteric_large_volume}
\Pker(E)\ge C_{V,\kernel,n}|E|^{2-\frac p\dimension}
\end{equation}
whenever $\chi_E\in\BVkerspace(\Rdim)$ with $|E|\ge V$, where $p=\log_2C$.
\end{proposition}

The assumption on the doubling constant in \cref{res:isoperimetric_large_volume} above is motivated by \cref{res:dou_comparison}. 
In the proof of \cref{res:isoperimetric_large_volume}, one needs to exploit that, given $p\in(\dimension,+\infty)$, there exists $C_{\dimension,p}>0$, such that, for any $x\in\Rdim$ and $E\in\measurablesets$ with $|E|\in(0,+\infty)$,
\begin{equation*}
\int_{E^c}\frac{\de y}{|x-y|^p}\ge C_{n,p}\,|E|^{1-\frac p\dimension}.
\end{equation*} 
For the simple proof of the above inequality, see~\cite{DPV12}*{Lem.~6.1} and~\cite{F21}*{Lem.~3.1}.

\subsection{Intersection with convex sets}

\label{subsec:convex}

The following result generalizes~\cite{FFMMM15}*{Lem.~B.1}, also see~\cite{CN22}*{Rem.~2.8}.

\begin{theorem}[Intersection with convex]\label{res:intersection_convex}
Let 
\eqref{H:Radial},
\eqref{H:Not_too_singular}, 
and 
\eqref{H:Decreasing_q} with $q=1$ 
be in force. 
If $E\in\measurablesets$ with $|E|<+\infty$, then 
$
\Pker(E\cap C)\leq\Pker(E)
$
for any convex  $C\subset\Rdim$.
\end{theorem}

As a consequence, we get the following non-local analog of the monotonicity of local perimeter, see~\cite{S18}*{Th.~1.1} and the references therein.

\begin{corollary}[Monotonicity on convex sets]
\label{res:archimede}
Let
\eqref{H:Radial},
\eqref{H:Not_too_singular}
and
\eqref{H:Decreasing_q} with $q=1$ 
be in force. 
If $A,B\in\measurablesets$, $A\subset B$, $|B|<+\infty$ and $A$ is convex, then
$
\Pker(A)\le\Pker(B).
$
\end{corollary}

\cref{res:intersection_convex}  follows from the following result, see~\cite{P20}*{Th.~1} and~\cite{C20}.

\begin{lemma}[Local minimality of half-spaces]
\label{res:pagliari}
Let~\eqref{H:Symmetric} and~\eqref{H:Far_from_zero_int} be in force. 
If $H\subset\Rdim$ is a half-space, then
$\Pker(H;B_R)\le\Pker(E;B_R)$
for $R>0$ and $E\in\measurablesets$ with $E\setminus B_R=H\setminus B_R$.
If also~\eqref{H:Positive} holds, then the inequality is strict for $E\ne H$. 
\end{lemma}

\begin{remark}[On \cref{res:pagliari}]
We warn the reader that~\cite{P20}*{Th.~1} is stated with the stronger \eqref{H:Not_too_singular} in place of \eqref{H:Far_from_zero_int}, see the discussion around~\cite{P20}*{Eq.~(2)}.
However, a careful inspection of the proofs of~\cite{P20}*{Ths.~1 and~2} allows to see that only the weaker \eqref{H:Far_from_zero_int} is really needed, while the stronger \eqref{H:Not_too_singular} additionally ensures that the class of competitors is sufficiently large, as in fact mentioned in~\cite{P20}*{Rem.~2}.
In passing, we also warn the reader that the pointwise convergence in~\cite{P20}*{Def.~1, point 2} has to be actually reinforced to an $\Lspace{1}$ convergence, as in~\cite{C20}*{Def.~2.1}.
We are indebted to Valerio Pagliari for having shared these observations on~\cites{P20,C20} with us.
\end{remark}

\begin{proof}[Proof of \cref{res:intersection_convex}]
The proof is almost identical to that of~\cite{FFMMM15}*{Lem.~B.1}, so we only sketch it.
First, by \cref{res:P_K_lsc}, one reduces to the case $C=H$ an half-space.
Second, since $|E|<+\infty$ and $x\mapsto\kernel(x)|x|$ is radially symmetric and decreasing due to \eqref{H:Radial} and \eqref{H:Decreasing_q} with $q=1$, by~\eqref{eq:Riesz_rearrang} one can also assume that $E\subset B_R$ for some $R>0$.
Hence
\begin{align*}
\Pker(E)-\Pker(E\cap H)
&=
\left(\int_{E^c}-\int_{E\cap H}\right)\int_{E\setminus H}\,K(x-y)\,\de x\,\de y
\\
&\ge
\left(\int_{F^c}-\int_{F\cap H}\right)\int_{F\setminus H}\,K(x-y)\,\de x\,\de y
\\
&=
\Pker(F;B_R)-\Pker(H;B_R),
\end{align*}
where $F=E\cup H$.
The conclusion hence follows by \cref{res:pagliari}.
\end{proof}

\section{\texorpdfstring{$(\kernel,\misurapeso)$}{(K,ν)}-Cheeger sets}

\label{sec:cheeger}

In this section, we study the theory of Cheeger sets for $\Pker$ and $\misurapeso\in\MisureAmmissibili$, where 
\begin{equation}
\label{eq:def_misure_ammissibili}
\MisureAmmissibili
=
\set*{
\misurapeso=\peso\lebdim :
w\in L^\infty(\Rdim),\  \essinf\limits_{\R^n} w>0
}
\end{equation}
denotes the set of \emph{admissible weight measures}.
To avoid heavy notation, when $\nu=\lebdim$
we simply drop the
reference to the measure.

\subsection{\texorpdfstring{$(\kernel,\misurapeso)$}{(K,ν)}-Cheeger sets}

For the general theory of Cheeger sets, see~\cites{P11,L15,FPSS22}.

\begin{definition}[$\kernel$-admissible sets]
\label{def:admissible_set}
A set $\Omega\in\measurablesets$ is \emph{$\kernel$-admissible} if there exists $E\in\measurablesets$ such that $E\subset\Omega$, $|E|\in(0,+\infty)$ and $\Pker(E)<+\infty$. 
\end{definition}

If~\eqref{H:Not_too_singular} holds, then any set $\Omega\subset\Rdim$ with non-empty interior is $\kernel$-admissible, since any non-empty open ball $B\subset\Omega$ satisfies $\Pker(B)<+\infty$ by~\eqref{eq:valerio_set}.
If $\kernel\in\Lspace{1}(\Rdim)$, then any $\Omega\in\measurablesets$ with $|\Omega|\in(0,+\infty)$ is $\kernel$-admissible by~\eqref{eq:L1_ker_perimeter}. 
Trivially, if \eqref{H:Decreasing_q} holds for $q\ge\dimension+1$, then no set $\Omega\in\measurablesets$ is admissible, because of \cref{res:constants}.
For this reason, in this section we assume that \eqref{H:Decreasing_q} holds with $q<\dimension+1$. 

\begin{definition}[$(\kernel,\misurapeso)$-Cheeger constant and $(\kernel,\misurapeso)$-Cheeger sets]
\label{def:Cheeger}
Given a $\kernel$-admissible set $\Omega\in\measurablesets$, we let
\begin{equation}\label{eq:Cheeger_const}
\Cheegkerpeso(\Omega)
=
\inf\set*{
\frac{\Pker(E)}{\misurapeso(E)}
:
E\in\measurablesets,\ E\subset\Omega,\ |E|\in(0,+\infty) 
}
\in[0,+\infty)
\end{equation}
be the \emph{$(\kernel,\misurapeso)$-Cheeger constant} of~$\Omega$. 
Any minimum $E$ in \eqref{eq:Cheeger_const} is a \textit{$(\kernel,\misurapeso)$-Cheeger set} of~$\Omega$ and we let~$\CheegerkerpesoClass(\Omega)$ be the collection of all $(\kernel,\misurapeso)$-Cheeger sets of~$\Omega$.
\end{definition}

The following result  generalizes~\cite{BLP14}*{Prop.~5.3} and is a particular application of~\cite{FPSS22}*{Th.~3.1} (also see~\cite{FPSS22}*{Sec.~7.3.1}), so we only sketch its proof.

\begin{theorem}[Existence of $(\kernel,\misurapeso)$-Cheeger sets]
\label{res:existence_Cheeger}
Let
\eqref{H:Radial},
\eqref{H:Far_from_zero_int},
\eqref{H:Not_integrable}
and
\eqref{H:Decreasing_q}
with
$q<\dimension+1$
be in force.
If $\Omega\in\measurablesets$ is a $\kernel$-admissible set with $|\Omega|<+\infty$, then $\CheegerkerpesoClass(\Omega)\ne\emptyset$, with $\Cheegkerpeso(\Omega)>0$ and
\begin{equation}
\label{eq:Cheeger_set_volume}
|E|^{\frac q\dimension-1}
\ge
\frac{\beta_\kernel(|\Omega|)}{\abovepeso\,|\Omega|^{2-\frac qn}\Cheegkerpeso(\Omega)}
\end{equation}
for all $E\in\CheegerkerpesoClass(\Omega)$, which, for $q=\dimension$, reduces to
\begin{equation}
\label{eq:faber-krahn_weak}
\Cheegkerpeso(\Omega)
\ge 
\frac{\beta_\kernel(|\Omega|)}{\abovepeso\,|\Omega|}.
\end{equation}
Moreover, if \eqref{H:Decreasing_dim_strinct} holds, $\Omega$ is open and $\nu=\lebdim$, then $\partial E\cap\partial\Omega\ne\emptyset$ for any $E\in\CheegerkerClass(\Omega)$.  
\end{theorem}

\begin{proof}
The existence part follows a plain compactness argument exploiting \cref{res:compactness_p}, \cref{res:isoperimetric},  \cref{res:behavior_beta} and \cref{res:P_K_lsc}.   
To prove~\eqref{eq:Cheeger_set_volume}, one sees that
\begin{equation}
\label{ghiaccio}
\misurapeso(E)
=
\frac{\Pker(E)}{\Cheegkerpeso(\Omega)}
\ge 
\frac{\Pker(B^{|E|})}{\Cheegkerpeso(\Omega)}
=
\frac{\beta_\kernel(|E|)}{\Cheegkerpeso(\Omega)}
\ge
\frac{\beta_\kernel(|\Omega|)}
{|\Omega|^{2-\frac q\dimension}\Cheegkerpeso(\Omega)}
\,
|E|^{2-\frac q\dimension}
\end{equation}
for any $E\in\CheegerkerpesoClass(\Omega)$ by \cref{res:isoperimetric} and \cref{res:monotonicity_isop_ratio}. 
Finally, letting $\Omega$ be open, $\misurapeso=\lebdim$ and  \eqref{H:Decreasing_dim_strinct} be in force, if $E\in\CheegerkerClass(\Omega)$ satisfies $E\Subset\Omega$ by contradiction, then also $tE\Subset\Omega$ for any $t>1$ sufficiently close to~$1$, so that
\begin{align*}
\frac{\Pker(tE)}{|tE|}
<
\frac{\Pker(E)}{|E|}
=
\Cheegker(\Omega),
\end{align*}
by \cref{res:monotonicity_isop_ratio},
in contrast with $E\in\CheegerkerClass(\Omega)$.
The proof is complete.
\end{proof}

\begin{remark}[On touching boundaries]
\label{rem:touch_boundary}
In  \cref{res:existence_Cheeger}, if \eqref{H:Decreasing_dim_strinct} holds, $\Omega$ is open and $\nu=\lebdim$, then  $E\in\CheegerkerClass(\Omega)$ must touch $\partial\Omega$. Actually, we proved that $tE\not\subset\Omega$ for   $t>1$.
\end{remark}

For general properties of $(\kernel,\misurapeso)$-Cheeger sets, see~\cite{FPSS22}*{Secs.~3.4, 3.5 and~7.3.1}. 
Here we only state the following result, generalizing~\cite{BLP14}*{Prop.~5.6}. 

\begin{proposition}[Relation with Euclidean Cheeger constant]
Let
\eqref{H:Not_too_singular}
be in force.
If $\Omega\subset\Rdim$ is an open set with $|\Omega|<+\infty$, then
\begin{equation*}
\Cheegkerpeso(\Omega)
\le 
\max\set*{\frac 1\belowpeso,\frac{h_{\misurapeso}(\Omega)}2}\int_\Rdim
(1\wedge|x|)\,\kernel(x)\,\de x,
\end{equation*}
where 
\begin{equation*}
h_\misurapeso(\Omega)
=
\inf\set*{
\frac{P(E)}{\misurapeso(E)}
:
E\in\measurablesets,\ E\subset\Omega,\ |E|>0 
}.
\end{equation*}
\end{proposition}

\subsection{\texorpdfstring{$(\kernel,\misurapeso)$}{(K,ν)}-calibrable sets and Faber--Krahn inequality}

A \textit{$(\kernel,\misurapeso)$-calibrable} set~$\Omega$ is a $\kernel$-admissible set with $\Omega\in\CheegerkerpesoClass(\Omega)$.
The following result proves that balls are $(K,\lebdim)$-calibrable, see~\cites{P11,L15} and~\cite{BLP14}*{Rem.~5.2}. 

\begin{proposition}[Balls are $K$-calibrable]
\label{res:calibrable}
Let
\eqref{H:Radial},
\eqref{H:Not_too_singular},
\eqref{H:Not_integrable},
\eqref{H:Decreasing_q}
with $q\in[n,n+1)$
be in force.
If $B$ is a (non-trivial) ball, then $B$ is $\kernel$-calibrable.
If also  \eqref{H:Decreasing_dim_strinct} holds, then $\CheegerkerClass(B)=\set*{B}$.
\end{proposition}

\begin{proof}
Let $B\subset\Rdim$ be a ball with $|B|\in(0,+\infty)$.
By~\eqref{eq:P_K_translation}, we can assume $B$ is centered at the origin.
By \eqref{H:Not_too_singular} and~\eqref{eq:valerio_set}, $B$ is $\kernel$-admissible.
If $E\subset B$ satisfies $|E|>0$, then 
\begin{equation}
\label{cicchetto}
\frac{\Pker(E)}{|E|}
\ge 
\frac{\Pker(B^{|E|})}{|E|}
=
\frac{\beta_\kernel(|E|)}{|E|}
\ge 
\frac{\beta_\kernel(|B|)}{|B|}
=
\frac{\Pker(B)}{|B|}
\end{equation}
because of \cref{res:isoperimetric} and \cref{res:monotonicity_isop_ratio} (with $q=n$).
By \cref{def:Cheeger}, this proves that $B\in\CheegerkerClass(B)$.
Actually, an inspection of~\eqref{cicchetto} yields $B^{|E|}\in\CheegerkerClass(B)$ for any $E\in\CheegerkerClass(B)$. 
Therefore, for $q=\dimension$, if also \eqref{H:Decreasing_dim_strinct} holds, then $tB^{|E|}\not\subset B$ for $t>1$ by \cref{rem:touch_boundary}, so $B^{|E|}=B$.
Thus $|E|=|B^{|E|}|=|B|$ and so $E=B$ as desired.
\end{proof}

The following result generalizes~\cite{BLP14}*{Prop.~5.5}, and~\eqref{eq:faber-krahn}  improves~\eqref{eq:faber-krahn_weak} for $\misurapeso=\lebdim$.

\begin{proposition}[$\kernel$-Faber--Krahn inequality]
\label{res:faber-krahn}
Let
\eqref{H:Radial},
\eqref{H:Not_too_singular},
\eqref{H:Not_integrable}
and
\eqref{H:Decreasing_q}
with $q\in[\dimension,\dimension+1)$
be in force.
If $\Omega\in\measurablesets$ is $\kernel$-admissible with $|\Omega|<+\infty$, then
\begin{equation}
\label{eq:faber-krahn}
\Cheegker(\Omega)
\ge 
\Cheegker(B^{|\Omega|}).
\end{equation}
If also \eqref{H:Decreasing_dim_strinct} holds, then \eqref{eq:faber-krahn} is an equality if and only if $\Omega$ is a ball.
\end{proposition}

\begin{proof}
Let $E\in\CheegerkerClass(\Omega)$ by \cref{res:existence_Cheeger}.
By \cref{res:isoperimetric} and \cref{res:monotonicity_isop_ratio}, we get
\begin{equation}
\label{kfk}
\Cheegker(\Omega)
=
\frac{\Pker(E)}{|E|}
\ge
\frac{\Pker(B^{|E|})}{|E|}
=
\frac{\beta_\kernel(|E|)}{|E|}
\ge
\frac{\beta_\kernel(|\Omega|)}{|\Omega|}
\,
\left(\frac{|E|}{|\Omega|}\right)^{1-\frac q\dimension}
\ge
\frac{\Pker(B^{|\Omega|})}{|B^{|\Omega|}|}
\ge
\Cheegker(B^{|\Omega|})
\end{equation}
proving~\eqref{eq:faber-krahn}.
If also \eqref{H:Decreasing_dim_strinct} holds (hence so \eqref{H:Radial_strict}) and if~\eqref{eq:faber-krahn} holds as an equality, then~\eqref{kfk} is a chain of equalities. 
In particular, $E$ coincides with $B^{|E|}$ up to a translation by \cref{res:isoperimetric}, and $B^{|E|}=B^{|\Omega|}$ by \cref{res:calibrable}.
Therefore $|E|=|B^{|E|}|=|B^{|\Omega|}|=|\Omega|$, so $E=\Omega$.
Thus $\Omega$ is equivalent to a ball up to a translation.
\end{proof}

\subsection{Relation with first eigenvalue}

Following~\cites{CL19,CC07,BLP14} and~\cite{FPSS22}*{Sec.~5}, we let
\begin{equation}
\label{eq:def_BV_0}
\BVkerspace_0(\Omega)
=
\set*{
u\in\BVkerspace(\Rdim)
:
u(x)=0\ \text{for $\lebdim$-a.e.}\ x\in\Rdim\setminus\Omega
}
\end{equation}
for $\Omega\in\measurablesets$.
By definition, $\BVkerspace_0(\Omega)$ is a closed subspace of $\BVkerspace(\Rdim)$.
By \cref{def:admissible_set} and \cref{res:coarea}, $\BVkerspace_0(\Omega)\ne\set*{0}$ if and only if $\Omega$ is $\kernel$-admissible.
Note that~\eqref{eq:def_BV_0}  circumvents any irregularity of $\partial\Omega$, see~\cite{CL19}*{Rem.~1.1} and~\cite{BLP14}*{Defs.~2.1 and~2.2 and Lem.~2.3}.

\begin{definition}[$(\kernel,\misurapeso)$-eigenvalue and eigenfunction]
\label{def:eigenvalue}
For $\Omega\in\measurablesets$ $\kernel$-ad\-mis\-si\-ble, we let
\begin{equation}
\label{eq:eigenvalue}
\lambda_{\kernel,\misurapeso}(\Omega)
=
\inf\set*{
\frac{\kerTV{u}}{\|u\|_{\Lspace{1}(\Rdim,\,\misurapeso)}}
:
u\in\BVkerspace_0(\Omega)\setminus\set{0
}
}
\in[0,+\infty)
\end{equation}
be the \emph{first $(\kernel,\misurapeso)$-eigenvalue} relative to $\Omega$.
Any function $u$ achieving the infimum in~\eqref{eq:eigenvalue} is called a \emph{$(\kernel,\nu)$-eigenfunction} of $\Omega$. 
\end{definition}

Since $|u|\in\BVkerspace_0(\Rdim)$ with $\kerTV{|u|}\le\kerTV{u}$ for all $u\in\BVkerspace_0(\Rdim)$, we have
\begin{equation}
\label{eq:eigenvalue_ge0}
\lambda_{\kernel,\misurapeso}(\Omega)
=
\inf\set*{
\kerTV{u}
:
u\in\BVkerspace_0(\Omega)\setminus\set{0},
\
\|u\|_{\Lspace{1}(\Rdim,\,\misurapeso)}=1,
\
u\ge0
}.
\end{equation}

\begin{remark}[$\kernel$-$1$-Laplacian]
The constant in~\eqref{eq:eigenvalue} is linked with the \emph{$\kernel$-$1$-Laplacian}
\begin{equation*}
\big\langle(-\Delta)^\kernel 
u,v\big\rangle
=
\frac12
\int_{\Rdim}
\int_{\Rdim}\frac{u(x)-u(y)}{|u(x)-u(y)|}\,(v(x)-v(y))\,\kernel(x-y)\,\de y\,\de x,
\quad
v\in \BVkerspace_0(\Rdim),
\end{equation*}
naturally arising from the expression of the $\BVkerspace$ energy in~\eqref{eq:def_kerTV},
as done for integrable kernels in~\cite{MRT19-b}*{Sec.~4.2}. 
We do not pursue this research direction here.
\end{remark}

The following result generalizes~\cite{BLP14}*{Th.~5.8} and is a particular case of~\cite{FPSS22}*{Th.~5.4} (for $N=1$), so we only sketch its proof.

\begin{theorem}[Relation with first $(\kernel,\misurapeso)$-eigenvalue]
\label{res:eigenvalue}
If $\Omega\in\measurablesets$ is a $\kernel$-admissible set with $|\Omega|<+\infty$, then
$\lambda_{\kernel,\misurapeso}(\Omega)
=
\Cheegkerpeso(\Omega)$.
Moreover, given $u\in\BVkerspace_0(\Omega)\setminus\set{0
}$ and letting 
\begin{equation}
\label{eq:levels_modified}
U_t
=
\begin{cases}
\set*{u>t}
&
\text{if}\ t\ge0,
\\[2mm]
\set*{u\le t}
&
\text{if}\ t<0,
\end{cases}
\end{equation} 
$u$ is a $(\kernel,\misurapeso)$-eigenfunction if and only if $U_t\in\CheegerkerpesoClass(\Omega)$ for a.e.\ $t\in\R$ such that $|U_t|>0$.
\end{theorem}

\begin{proof}
Since $\BVkerspace_0(\Omega)\ne\set*{0}$, $\Cheegker(\Omega)
\ge
\lambda_{\kernel,\misurapeso}(\Omega)$ according to Definitions~\ref{def:Cheeger} and~\ref{def:eigenvalue}.
For the converse inequality, if $u\in\BVkerspace_0(\Omega)\setminus\set{0}$, then the sets in~\eqref{eq:levels_modified} satisfy $U_t\in\measurablesets$ for $\lebone$-a.e.\ $t\in\R$ with $U_t\subset\Omega$.
In addition, thanks to~\eqref{eq:complement_invariance} and \cref{res:coarea}, we get  
\begin{align*}
\kerTV{u}
=
\int_\R\Pker(\set*{u>t})\,\de t
=
\int_\R\Pker(U_t)\,\de t.
\end{align*} 
By Cavalieri's formula, we have 
$\|u\|_{\Lspace{1}(\Rdim,\,\misurapeso)}=\int_{\R}\misurapeso(U_t)\,dt$.
Therefore, $\Pker(U_t)<+\infty$ and $|U_t|<+\infty$ for $\lebone$-a.e.\ $t\in\R$, so
$\frac{\Pker(U_t)}{\misurapeso(U_t)}
\ge
\Cheegkerpeso(\Omega)$
whenever $|U_t|>0$.
We thus have that
\begin{equation}
\label{eq:per-gatti}
\frac{\kerTV{u}}{\|u\|_{\Lspace{1}(\Rdim,\,\misurapeso)}}
=
\frac{1}{\|u\|_{\Lspace{1}(\Rdim,\,\misurapeso)}}
\int_\R\frac{\Pker(U_t)}{\misurapeso(U_t)}\,\misurapeso(U_t)\,\de t
\ge
\frac{\Cheegkerpeso(\Omega)}{\|u\|_{\Lspace{1}(\Rdim,\,\misurapeso)}}
\int_\R\misurapeso(U_t)\,\de t
=
\Cheegkerpeso(\Omega),
\end{equation}
proving $\lambda_{\kernel,\misurapeso}(\Omega)\ge\Cheegkerpeso(\Omega)$.
Finally, if $u\in\BVkerspace_0(\Omega)$ is a $(\kernel,\misurapeso)$-eigenfunction of $\Omega$, then
\begin{equation}
\label{eq:trippa}
\int_\R\Pker(U_t)-\Cheegkerpeso(\Omega)\,\misurapeso(U_t)\,\de t=0.
\end{equation}
Since $\Pker(U_t)\ge \Cheegkerpeso(\Omega)\,\misurapeso(U_t)$ for $\lebone$-a.e.\ $t\in\R$,
we must have $\Pker(U_t)=\Cheegkerpeso(\Omega)\,\misurapeso(U_t)$ for $\lebone$-a.e.\ $t\in\R$ such that $|U_t|>0$, yielding $U_t\in\CheegerkerpesoClass(\Omega)$.
Viceversa, if $U_t\in\CheegerkerpesoClass(\Omega)$ for $\lebone$-a.e.\ $t\in\R$ such that $|U_t|>0$, then~\eqref{eq:trippa} is true and so~\eqref{eq:per-gatti} holds as an equality, yielding the minimality of $u$ in~\eqref{eq:eigenvalue}, i.e., $u$ is a $(\kernel,\misurapeso)$-eigenfunction of $\Omega$.
\end{proof}

The following result is a non-local analog of~\cite{CC07}*{Th.~4} and generalizes~\cite{BLP14}*{Th.~7.1}.

\begin{corollary}[$\Lspace{\infty}$ bound]
\label{res:L_infty-bound}
Let 
\eqref{H:Radial}
and
\eqref{H:Decreasing_q}
with $q\in(\dimension,\dimension+1)$
be in force and let $\Omega\in\measurablesets$ be a $\kernel$-admissible set with $|\Omega|<+\infty$.
If $u\in\BVkerspace_0(\Omega)$ is a $(\kernel,\misurapeso)$-eigenfunction of $\Omega$, then $u\in\Lspace{\infty}(\Omega)$ with
\begin{equation}
\label{eq:L_infty-bound}
\|u\|_{\Lspace{\infty}(\Omega)}
\le 
\left(
\frac{\abovepeso\,|\Omega|^{2-\frac qn}\Cheegkerpeso(\Omega)}
{\beta_\kernel(|\Omega|)}
\right)^{\frac\dimension{q-\dimension}}
\,
\|u\|_{\Lspace{1}(\Omega)}.
\end{equation}
\end{corollary}

\begin{proof}
The proof is similar to~\cite{BLP14}*{Rem.~7.3}, so we only sketch it.
Let $u\in\BVkerspace_0(\Omega)$ be a $(\kernel,\misurapeso)$-eigenfunction of $\Omega$ such that $u\ge0$ in view of~\eqref{eq:eigenvalue_ge0}.
Letting $U_t=\set*{u>t}$ for all $t\ge0$ as in~\eqref{eq:levels_modified}, by \cref{res:eigenvalue} we get $U_t\in\CheegerkerpesoClass(\Omega)$ for a.e.\ $t\in[0,\|u\|_{\Lspace{\infty}(\Omega)})$.
Arguing as in the proof of \cref{res:existence_Cheeger} to infer~\eqref{ghiaccio}, by \cref{res:isoperimetric} and \cref{res:monotonicity_isop_ratio} we obtain~\eqref{eq:Cheeger_set_volume} for $U_t$, namely, since $q>\dimension$, 
\begin{equation*}
|U_t|
\ge
\left(
\frac{\beta_\kernel(|\Omega|)}{\abovepeso\,|\Omega|^{2-\frac qn}\Cheegkerpeso(\Omega)}
\right)^{\frac\dimension{q-\dimension}}
\end{equation*}
for a.e.\ $t\in[0,\|u\|_{\Lspace{\infty}(\Rdim)})$.
Integrating the above inequality for $t\in[0,\|u\|_{\Lspace{\infty}(\Rdim)})$, we get
\begin{equation*}
\|u\|_{\Lspace{1}(\Omega)}
\ge 
\int_0^{\|u\|_{\Lspace{\infty}(\Omega)}}
\left(
\frac{\beta_\kernel(|\Omega|)}{\abovepeso\,|\Omega|^{2-\frac qn}\Cheegkerpeso(\Omega)}
\right)^{\frac\dimension{q-\dimension}}
\,\de t
=
\left(
\frac{\beta_\kernel(|\Omega|)}{\abovepeso\,|\Omega|^{2-\frac qn}\Cheegkerpeso(\Omega)}
\right)^{\frac\dimension{q-\dimension}}
\,
\|u\|_{\Lspace{\infty}(\Omega)}
\end{equation*}
from which~\eqref{eq:L_infty-bound} readily follows.
\end{proof}

\subsection{A non-local Max Flow Min Cut Theorem}

The Min Cut Max Flow Theorem~\cite{G06} asserts that the (classical) Cheeger constant of a (sufficiently regular) set $\Omega$ can be recovered by minimizing the $\Lspace{\infty}$ norm of vector fields with prescribed divergence on~$\Omega$. 
Analogous results are known in the fractional~\cite{BLP14}*{Sec.~8} and integrable~\cite{MRT19-b}*{Sec.~5.1} cases.
We shall now prove an analogous result in the present setting.

Mimicking~\cite{BLP14}*{Sec.~8} and~\cite{B11}*{Sec.~9.4}, given $\Omega\in\measurablesets$, we let 
\begin{equation*}
\BVkerdual(\Omega)
=
\set*{F\colon\BVkerspace_0(\Omega)\to\R : F\ \text{linear and continuous}}
\end{equation*}
be the \emph{topological dual space }of $\BVkerspace_0(\Omega)$.
Since $\BVkerspace_0(\Omega)\subset L^1(\Rdim)$, we have $L^\infty(\Rdim)\subset \BVkerdual(\Omega)$ continuously, with the natural action on $\BVkerspace_0(\Omega)$ (tacitly used below).  
We thus let $\Rkernel\colon\BVkerspace_0(\Omega)\to\Lspace{1}(\Rdim\times\Rdim)$ be defined by 
\begin{equation*}
\Rkernel(u)(x,y)=\frac12\,(u(x)-u(y))\,
\kernel(x-y),
\quad
(x,y)\in\Rdim\times\Rdim.
\end{equation*}
Note that $\Rkernel$ is well posed, since
\begin{equation}\label{eq:norma_Rkernel}
\|\Rkernel(u)\|_{\Lspace{1}(\Rdim\times\Rdim)}=
\kerTV{u}
\quad
\text{for}\
u\in\BVkerspace_0(\Omega).
\end{equation}

Let $\Rkerneldual\colon\Lspace{\infty}(\Rdim\times\Rdim)\to\BVkerdual(\Omega)$ be the \textit{adjoint operator} of $\Rkernel$, given by
\begin{equation*}
\pairing{\Rkerneldual(\phi),u}_{(\BVkerdual(\Omega),\,\BVkerspace_0(\Omega))}
=
\int_{\Rdim}\int_{\Rdim}\phi\,\Rkernel(u)\,\de x\,\de y,
\end{equation*}
whenever $\phi\in\Lspace{\infty}(\Rdim\times\Rdim)$ and $u\in\BVkerspace_0(\Omega)$.
Roughly speaking, the operator $\Rkerneldual$ can be thought of as a sort of non-local $\kernel$-divergence formally acting on $\Lspace{\infty}$ functions as
\begin{equation*}
\Rkerneldual(\phi)(x)
=
\int_\Rdim(\phi(x,y)-\phi(y,x))\,\kernel(x-y)\,\de y,
\quad
x\in\Rdim,
\end{equation*}
whenever $\phi\in\Lspace{\infty}(\Rdim\times\Rdim)$.
In particular, symmetric $\Lspace{\infty}$
functions correspond to divergence-free vector fields.

The following result generalizes~\cite{BLP14}*{Th.~8.6} and the second part of~\cite{MRT19-b}*{Th.~5.3}.

\begin{theorem}[$\kernel$-Max Flow Min Cut Theorem]
\label{res:non-local_Max_Flow_Min_Cut}
Let $\Omega\in\measurablesets$ be a $\kernel$-admissible set with $|\Omega|<+\infty$ and let $\misurapeso=\peso\lebdim \in\MisureAmmissibili$.
Then,
\begin{equation}
\label{eq:non-local_Max_Flow_Min_Cut}
\frac{1}{\Cheegkerpeso(\Omega)}
=
\min\set*{\|\phi\|_{\Lspace{\infty}(\Rdim\times\Rdim)}:\phi\in\Lspace{\infty}(\Rdim\times\Rdim)\ \text{such that}\ \Rkerneldual(\phi)=\peso\chi_{\Omega}},
\end{equation}
where the minimum in~\eqref{eq:non-local_Max_Flow_Min_Cut} equals $+\infty$ if and only if $\Rkerneldual(\phi)\neq\peso\chi_{\Omega}$ for $\phi\in\Lspace{\infty}(\Rdim\times\Rdim)$, otherwise it is finite and achieved by some $\phi\in\Lspace{\infty}(\Rdim\times\Rdim)$ such that $\Rkerneldual(\phi)=\peso\chi_{\Omega}$.
\end{theorem}

\begin{proof}
The proof is similar to that of~\cite{BLP14}*{Th.~8.6}, so we only sketch it.
By \cref{res:eigenvalue},
\begin{align*}
(0,+\infty]
\ni
\frac{1}{\Cheegkerpeso(\Omega)}
&=
\sup\set*{\frac{\|u\|_{\Lspace{1}(\Rdim,\,\misurapeso)}}{\kerTV{u}}
			:
			u\in\BVkerspace_0(\Omega)\setminus\set{0
			}
		}\\
	&=
	\sup\set*{
		\pairing*{\peso\chi_{\Omega},u}
		:
		u\in\BVkerspace_0(\Omega),\  \kerTV{u}\leq 1
	}.
	\end{align*}
We can rewrite the latter $\sup$ as
\begin{equation*}
\sup\set*{
		\pairing*{\peso\chi_{\Omega},u}
		:
		u\in\BVkerspace_0(\Omega),\  \kerTV{u}\leq 1
	}
=
\sup\set*{\pairing*{\bfx^*,\bfx}-\Gekeland(\Aekeland(\bfx)): \bfx\in\Xekeland}
\end{equation*}
where $\Xekeland=\BVkerspace_0(\Omega)$, $\bfx^*=\peso\chi_{\Omega}\in\Xekeland^*$, $\Aekeland=\Rkernel$ and
\begin{equation*}
\Gekeland(\bfy)=
\begin{cases}
0 & \text{if}\ \|\bfy\|_{\Yekeland}\leq 1,
\\[.5ex]
+\infty & \text{otherwise},
\end{cases}
\end{equation*}
for $\bfy\in\Yekeland=\Lspace{1}(\Rdim\times\Rdim)$. 
Since $\Gekeland\colon\Yekeland\to\R\cup\set{+\infty}$ is convex and lower semicontinuous and $\Aekeland\colon\Xekeland\to\Yekeland$ is linear and continuous again by~\eqref{eq:norma_Rkernel},  we get that 
\begin{equation}
\label{eq:ekeland}
\begin{split}
\sup\set*{\pairing*{\bfx^*,\bfx}-\Gekeland(\Aekeland(\bfx)): \bfx\in\Xekeland}
&=
(\Gekeland\circ\Aekeland)^*(\bfx^*)
\\
&=
\min\set*{\Gekeland^*(\bfy^*):\bfy^*\in\Yekeland^*\ \text{such that}\   \Aekeland^*(\bfy^*)=\bfx^*},
\end{split}
\end{equation}
by~\cite{E90}*{Prop.~5 in Ch.~II, Sec.~2},
where $(\Gekeland\circ\Aekeland)^*\colon\Xekeland^*\to\R\cup\set*{+\infty}$ and $\Gekeland^*\colon\Yekeland^*\to\R\cup\set*{+\infty}$ are  the \textit{Fenchel conjugates} of $\Gekeland\circ\Aekeland$ and $\Gekeland$, see~\cite{E90}*{Def.~7 in Ch.~II, Sec.~1} and $\Aekeland^*=\Rkernel^*$.
The minimum in~\eqref{eq:ekeland} equals $+\infty$ if and only if ${(\Aekeland^*)}^{-1}(\set{x^*})=\emptyset$, otherwise it is finite and there exists some point in ${(\Aekeland^*)}^{-1}(\set{x^*})$ achieving the minimum. 
To conclude, we simply note that  
$\Yekeland^*=\Lspace{\infty}(\Rdim\times\Rdim)$
and
$\Gekeland^*(\bfy^*)=\|\bfy^*\|_{\Lspace{\infty}(\Rdim\times\Rdim)}$
for $\bfy^*\in\Yekeland^*$. 
\end{proof}

As a consequence, we get the following characterization, which generalizes~\cite{BLP14}*{Cor.~8.7} and the second part of~\cite{MRT19-b}*{Th.~5.3}. 
Below, we set $\frac{1}{+\infty}=0$ and $\max\emptyset=0$.

\begin{corollary}
\label{res:cut-leq}
Let $\Omega\in\measurablesets$ be a $\kernel$-admissible set with $|\Omega|<+\infty$ and let $\misurapeso=\peso\lebdim \in\MisureAmmissibili$. Then,
\begin{equation}
\label{eq:non-local_Max_Flow_Min_Cut_Corollary}	
\Cheegkerpeso(\Omega)
=
\max\big\{h\in\R:
\exists\,\phi\in\Lspace{\infty}(\Rdim\times\Rdim)\ \text{with}\ \|\phi\|_{\Lspace{\infty}(\Rdim\times\Rdim)}\leq 1\ 
\text{and}\ \Rkerneldual(\phi)\geq h\peso\chi_{\Omega}
\big\},
\end{equation}
where $\Rkerneldual(\phi)\geq h\peso\chi_{\Omega}$ means that
$\pairing*{\Rkerneldual(\phi),u}
\geq 
h\int_{\Omega}u\,\de \misurapeso$
for 
$u\in\BVkerspace_0(\Omega)$ with $u\geq 0$.
Furthermore, $\Cheegkerpeso(\Omega)=0$ if and only if the maximization set in~\eqref{eq:non-local_Max_Flow_Min_Cut_Corollary} is empty.  
\end{corollary}

\begin{proof}
The proof is similar to that of~\cite{BLP14}*{Cor.~8.7}, so we only sketch it. 
By~\eqref{eq:non-local_Max_Flow_Min_Cut},\begin{equation}
\label{eq:ekeland_sotto}
[0,+\infty)\ni
\Cheegkerpeso(\Omega)
=
\max\set*{
\tfrac{1}{\|\phi\|_{\Lspace{\infty}(\Rdim\times\Rdim)}}
:
\phi\in\Lspace{\infty}(\Rdim\times\Rdim)\
\text{such that}\ 
\Rkerneldual(\phi)= \peso\chi_{\Omega}}
\end{equation}
where the $\max$ is $0$ if and only if the set is empty, otherwise the $\max$ is achieved by some $\phi\in\Lspace{\infty}(\Rdim\times\Rdim)$, with $\|\phi\|_{\Lspace{\infty}(\Rdim\times\Rdim)}\in(0,+\infty)$, such that $\Rkerneldual(\phi)=w\chi_\Omega$.
We note that
\begin{align*}
\max
&
\set*{
\tfrac{1}{\|\phi\|_{\Lspace{\infty}(\Rdim\times\Rdim)}}
:
\phi\in\Lspace{\infty}(\Rdim\times\Rdim)\
\text{such that}\ 
\Rkerneldual(\phi)= \peso\chi_{\Omega}}
\\[2mm]
&=
\max\set*{h\in\R:\exists \,\phi\in\Lspace{\infty}(\Rdim\times\Rdim)\ \text{with}\ 
\tfrac{1}{\|\phi\|_{\Lspace{\infty}(\Rdim\times\Rdim)}}\geq h\
\text{and}\ \Rkerneldual(\phi)\geq \peso\chi_{\Omega}},
\end{align*}
where any of the two $\max$ above is equal to $0$ if and only if the corresponding set is empty, otherwise both $\max$ are achieved. 
Indeed, recalling the way the equality in~\eqref{eq:ekeland_sotto} is understood, in the non-empty case (otherwise being easier and thus omitted), we have
\begin{equation*}
\max\set*{
\tfrac{1}{\|\phi\|_{\Lspace{\infty}(\Rdim\times\Rdim)}}
:
\phi\in\Lspace{\infty}(\Rdim\times\Rdim)\
\text{such that}\ 
\Rkerneldual(\phi)= \peso\chi_{\Omega}}
=
\tfrac{1}{\|\Phi\|_{\Lspace{\infty}(\Rdim\times\Rdim)}}
\end{equation*} 
for some $\Phi\in\Lspace{\infty}(\Rdim\times\Rdim)$ such that $\Rkerneldual(\Phi)= \peso\chi_{\Omega}$, so that 
\begin{align*}
&\tfrac{1}{\|\Phi\|_{\Lspace{\infty}(\Rdim\times\Rdim)}}
=
\max\set*{h\in\R:
h\le 
\tfrac{1}{\|\Phi\|_{\Lspace{\infty}(\Rdim\times\Rdim)}}
}
\\
&\quad=
\max\set*{h\in\R:\exists \,\phi\in\Lspace{\infty}(\Rdim\times\Rdim)\ \text{with}\ 
\tfrac{1}{\|\phi\|_{\Lspace{\infty}(\Rdim\times\Rdim)}}\geq h\
\text{and}\ \Rkerneldual(\phi)\geq \peso\chi_{\Omega}},
\end{align*}
yielding the conclusion.
\end{proof}

\section{The functional problem}

\label{sec:functional_pb}

In this section, we study the functional $\kernel$-variation denoising model.

\subsection{Definition of the problem}

Let $\kernel\colon\Rdim\to[0,+\infty]$ be a kernel, $\Lambda>0$, $f\in\Llocspace{1}(\Rdim)$ and $\misurapeso=w\lebdim\in\MisureAmmissibili$, where $\MisureAmmissibili$ is defined in~\eqref{eq:def_misure_ammissibili}.
To avoid heavy notation, we drop the reference to the measure when $\misurapeso=\lebdim$. 
We consider the \emph{functional energy}
\begin{align*}		
\EnergyLOneKerTV{(u)}
=
\EnergyLOneKerTV{(u;f,\Lambda,\misurapeso)}
=
\kerTV{u}+\Lambda\int_{\Rdim}|u-f|\,\de \misurapeso, 
\quad
u\in\Llocspace{1}(\Rdim),
\end{align*}
and the associated minimization problem
\begin{equation}
\label{Pb:FunctionalProblem}
\tag{$\FunctionalProblem{f,\Lambda}{\misurapeso}$}
\inf\set*{
\EnergyLOneKerTV{(u;f,\Lambda,\misurapeso)}
:
u\in\Llocspace{1}(\Rdim)
}.
\end{equation}
We say that $u\in\Llocspace{1}(\Rdim)$ is a \emph{solution} of~\eqref{Pb:FunctionalProblem} provided that
$\EnergyLOneKerTV{(u;f,\Lambda,\misurapeso)}<+\infty$
and
$\EnergyLOneKerTV{(u;f,\Lambda,\misurapeso)}
\le 
\EnergyLOneKerTV{(v;f,\Lambda,\misurapeso)}$
for $v\in\Llocspace{1}(\Rdim)$, and we let $\FunctionalProblemSolutions{f,\Lambda,\misurapeso}$ be the set of solutions of~\eqref{Pb:FunctionalProblem}.
Note that $\FunctionalProblemSolutions{f,\Lambda,\misurapeso}\subset\BVkerlocspace(\Rdim)$ and $\FunctionalProblemSolutions{f,\Lambda,\misurapeso}\subset\BVkerspace(\Rdim)$ for $f\in\Lspace{1}(\Rdim)$.

\subsection{Properties of solutions}

We collect some properties of $\FunctionalProblemSolutions{f,\Lambda,\misurapeso}$, as in~\cite{B22}.

\begin{proposition}[Basic properties of $\FunctionalProblemSolutions{f,\Lambda,\misurapeso}$]
\label{res:basic_props_Sol}
The following hold:
\begin{enumerate}[label=(\roman*)]

\item\label{item:convex_and_closed} 
$\FunctionalProblemSolutions{f,\Lambda,\misurapeso}$ is a convex and closed set in $\Llocspace{1}(\Rdim)$;

\item\label{item:stability} 
if  $u_k\in\FunctionalProblemSolutions{f_k,\Lambda,\misurapeso}$, $f_k\to f$ in $L^1(\Rdim)$ and $u_k\to u$ in $\Llocspace{1}(\Rdim)$ as $k\to+\infty$, then $u\in\FunctionalProblemSolutions{f,\Lambda,\misurapeso}$;

\item\label{item:translation} 
$\FunctionalProblemSolutions{f+c,\Lambda,\misurapeso}=\FunctionalProblemSolutions{f,\Lambda,\misurapeso}+c$ for $c\in\R$;
		\item\label{item:dilation} $\lambda\,\FunctionalProblemSolutions{f,\Lambda,\misurapeso}=\FunctionalProblemSolutions{\lambda f,\Lambda,\misurapeso}$ for $\lambda\in\R\setminus \{0\}$; 
		\item\label{item:truncation} 
if $u\in\FunctionalProblemSolutions{f,\Lambda,\misurapeso}$, then $u^+\in\FunctionalProblemSolutions{f^+,\Lambda,\misurapeso}$ and $u^-\in\FunctionalProblemSolutions{f^-,\Lambda,\misurapeso}$;
		\item\label{item:truncationbis} 
if $u\in\FunctionalProblemSolutions{f,\Lambda,\misurapeso}$, then $u\wedge c\in\FunctionalProblemSolutions{f\wedge c,\Lambda,\misurapeso}$ and $u\vee c\in\FunctionalProblemSolutions{f\vee c,\Lambda,\misurapeso}$ for $c\in\R$.

\end{enumerate}
\end{proposition}

\subsection{Existence for \texorpdfstring{$\Lspace{1}$}{Lˆ1} data}

The following result proves the existence of solutions of~\eqref{Pb:FunctionalProblem} for $f\in\Lspace{1}(\Rdim)$ as a consequence of \cref{res:BV_K_lsc} and \cref{res:compactness_p}.

\begin{proposition}[Existence for \eqref{Pb:FunctionalProblem} with $f\in\Lspace{1}$]\label{prop:existenceL1Datum}
Let 
\eqref{H:Far_from_zero_int} and
\eqref{H:Not_integrable}
be in force.
If $f\in\Lspace{1}(\Rdim)$, then $\FunctionalProblemSolutions{\misurapeso,f,\Lambda}\ne\emptyset$.
\end{proposition}

\section{The geometric problem}
\label{sec:geometric_pb}

In this section, we study the geometric $\kernel$-variation denoising model.
The results below can be proved as in~\cite{B22} with minor adjustments (even if the datum set has finite measure).

\subsection{Definition of the problem}

Let $\Lambda>0$, $E\in\measurablesets$ and $\misurapeso\in\MisureAmmissibili$. 
We consider the \emph{geometric energy}
\begin{equation*}		
\EnergyGeometricKer{(U;E,\Lambda,\misurapeso)}
=
\Pker(U)+\Lambda\,\misurapeso\tonde{E\bigtriangleup U},
\quad
U\in\measurablesets, 
\end{equation*}
and the associated minimization problem
\begin{equation}
\label{Pb:GeometricProblem}
\tag{$\GeometricProblem{E,\Lambda}{\misurapeso}$}
\inf
\set*{
\EnergyGeometricKer{(U;E,\Lambda,\misurapeso)}
:
U\in\measurablesets
}.
\end{equation}
We say that $U\in\measurablesets$ is a \emph{solution} (or a \emph{global minimum}) of~\eqref{Pb:GeometricProblem} if 
$\EnergyGeometricKer(U;E,\Lambda,\misurapeso)<+\infty$
and
$\EnergyGeometricKer{(U;E,\Lambda,\misurapeso)}\leq \EnergyGeometricKer{(V;E,\Lambda,\misurapeso)}$
for all  
$V\in\measurablesets$,
and we let $\GeometricProblemSolutions{E,\Lambda,\misurapeso}$ be the set of solutions of~\eqref{Pb:GeometricProblem}.
Note that, if $F\in\GeometricProblemSolutions{E,\Lambda,\misurapeso}$, then $\chi_F\in\BVkerlocspace(\Rdim)$, and, analogously, if $F\in\GeometricProblemSolutions{E,\Lambda,\misurapeso}$ then $\chi_F\in \BVkerspace(\Rdim)$ whenever  $|E|<+\infty$.

The following result is a simple consequence of \cref{res:coarea} and of the \emph{layer-cake formula}
\begin{equation*}
\|f-g\|_{\Lspace{1}(\Rdim,\misurapeso)}
=
\int_\R\misurapeso\tonde{\set*{f>t}\bigtriangleup\set*{g>t}}\,\de t,
\quad
\text{for}\
f,g\in\Lspace{1}(\Rdim).
\end{equation*}

\begin{lemma}[Layer-cake formula]
\label{res:layer-cake}
If $f\in\Lspace{1}(\Rdim)$ and $u\in\BVkerspace(\Rdim)$, then
\begin{equation*}
\EnergyLOneKerTV{(u;f,\Lambda,\misurapeso)}
=
\int_\R
\EnergyGeometricKer{(\set*{u>t};\set*{f>t},\Lambda,\misurapeso)}\,\de t.
\end{equation*}
\end{lemma}

\subsection{Properties of solutions}

We collect some properties of $\GeometricProblemSolutions{E,\Lambda,\misurapeso}$, as in~\cite{B22}.

\begin{proposition}[Basic properties of solutions]
\label{res:basic_props_GSol}
Let $\Lambda>0$, $E\in\measurablesets$ and $\misurapeso\in\MisureAmmissibili$, $f\in\Lspace{1}(\Rdim)$ and $u\in\BVkerspace(\Rdim)$. 
The following hold:
\begin{enumerate}[label=(\roman*)]

\item\label{item:g_translation}
if $U\in\GeometricProblemSolutions{E,\Lambda,\misurapeso}$, then $U+x\in\GeometricProblemSolutions{E+x,\Lambda,\misurapeso_x}$ for $x\in\Rdim$, where $\misurapeso_x(A)=\misurapeso(A-x)$ for  $A\in\measurablesets$;

\item\label{item:g_stability} 
if $U_k\in\GeometricProblemSolutions{E_k,\Lambda,\misurapeso}$ with $E_k\in\measurablesets$ for $k\in\N$ and $\chi_{E_k}\to\chi_E$ in $\Lspace{1}(\Rdim)$ and $\chi_{U_k}\to\chi_U$ in $\Llocspace{1}(\Rdim)$ as $k\to+\infty$, then $U\in\GeometricProblemSolutions{E,\Lambda,\misurapeso}$;

\item\label{item:complement} 
if $U\in\GeometricProblemSolutions{E,\Lambda,\misurapeso}$ then $\comp{U}\in\GeometricProblemSolutions{\comp{E},\Lambda,\misurapeso}$;

\item\label{item:fromGPtoP} 
if $\set*{u>t}\in\GeometricProblemSolutions{\set*{f>t},\Lambda,\misurapeso}$ for $\lebone$-a.e.\ $t\in\R$, then $u\in\FunctionalProblemSolutions{f,\Lambda,\misurapeso}$;

\item\label{item:fromPtoGP} 
if	$u\in\FunctionalProblemSolutions{f,\Lambda,\misurapeso}$, then $\set*{u>t}\in\GeometricProblemSolutions{\set*{f>t},\Lambda,\misurapeso}$ for $t\in\R\setminus\set*{0}$.

\end{enumerate}
\end{proposition}

\begin{remark}[Other versions of \cref{res:basic_props_GSol}\ref{item:fromGPtoP} and~\ref{item:fromPtoGP}]
\label{rem:other_versions_t}
As observed in~\cite{B22},     properties~\ref{item:fromGPtoP} and~\ref{item:fromPtoGP} in \cref{res:basic_props_GSol} can be completed with the following:
\begin{enumerate}[label=(\roman*$^{1}$)]
\setcounter{enumi}{2}

\item
if $\set*{u\ge t}\in\GeometricProblemSolutions{\set*{f\ge t},\Lambda,\misurapeso}$ for $\lebone$-a.e.\ $t\in\R$, then $u\in\FunctionalProblemSolutions{f,\Lambda,\misurapeso}$;

\item
if	$u\in\FunctionalProblemSolutions{f,\Lambda,\misurapeso}$, then $\set*{u\ge t}\in\GeometricProblemSolutions{\set*{f\ge t},\Lambda,\misurapeso}$ for $t\in\R\setminus\set*{0}$;

\end{enumerate}

\begin{enumerate}[label=(\roman*$^{2}$)]
\setcounter{enumi}{2}

\item
if $\set*{u\le t}\in\GeometricProblemSolutions{\set*{f\le t},\Lambda,\misurapeso}$ for $\lebone$-a.e.\ $t\in\R$, then $u\in\FunctionalProblemSolutions{f,\Lambda,\misurapeso}$;

\item
if	$u\in\FunctionalProblemSolutions{f,\Lambda,\misurapeso}$, then $\set*{u\le t}\in\GeometricProblemSolutions{\set*{f\le t},\Lambda,\misurapeso}$ for $t\in\R\setminus\set*{0}$;

\end{enumerate}

\begin{enumerate}[label=(\roman*$^{3}$)]
\setcounter{enumi}{2}

\item
if $\set*{u< t}\in\GeometricProblemSolutions{\set*{f< t},\Lambda,\misurapeso}$ for $\lebone$-a.e.\ $t\in\R$, then $u\in\FunctionalProblemSolutions{f,\Lambda,\misurapeso}$;

\item
if	$u\in\FunctionalProblemSolutions{f,\Lambda,\misurapeso}$, then $\set*{u< t}\in\GeometricProblemSolutions{\set*{f< t},\Lambda,\misurapeso}$ for $t\in\R\setminus\set*{0}$.

\end{enumerate}
\end{remark}

\begin{remark}[Case $t=0$ in \cref{res:basic_props_GSol} and \cref{rem:other_versions_t}]
\label{rem:zero_t}
The case $t=0$ in \cref{res:basic_props_GSol} and \cref{rem:other_versions_t} is more delicate, as $\set*{u>0}\in\GeometricProblemSolutions{\set*{f>0},\Lambda,\misurapeso}$ if either $|\set*{f>0}|<+\infty$ or $|\set*{f\le 0}|<+\infty$, and similarly for $\set*{u\ge0}$, $\set*{u<0}$ and $\set*{u\le0}$.
\end{remark}

From \cref{res:basic_props_Sol}, \cref{rem:other_versions_t} and \cref{rem:zero_t}, we can deduce the following result.   

\begin{corollary}[Geometric $\Lspace{1}$ datum]\label{cor:BINequivalenceFuncGeomPBs}
Let $E\in\measurablesets$ be such that $|E|<+\infty$.

\begin{enumerate}[label=(\roman*)]
		\item\label{item:BINfromGPtoP} 
If $U\in\GeometricProblemSolutions{E,\Lambda,\misurapeso}$, then $\chi_{U}\in\FunctionalProblemSolutions{\chi_{E},\Lambda,\misurapeso}$.
		\item\label{item:BINfromPtoGP}  		
If	$u\in\FunctionalProblemSolutions{\chi_{E},\Lambda,\misurapeso}$, then $u(x)\in[0,1]$ for $\lebdim$-a.e. $x\in\Rdim$, with $\{u>t\}\in\GeometricProblemSolutions{E,\Lambda,\misurapeso}$ for all $t\in[0,1)$ and $\{u\geq t\}\in\GeometricProblemSolutions{E,\Lambda,\misurapeso}$ for all $t\in(0,1]$.
\end{enumerate}
\end{corollary}

As a consequence, we get the following result dealing with (countable) intersections and unions of solutions of~\eqref{Pb:GeometricProblem}.

\begin{corollary}[Intersection and union]
\label{res:GP_intersection_union}
Let $E\in\measurablesets$ be such that $|E|<+\infty$.

\begin{enumerate}[label=(\roman*)]

\item\label{item:GP_intersection_union} 
$\GeometricProblemSolutions{E,\Lambda,\misurapeso}$ is closed under finite intersection and finite union.

\item\label{item:GP_intersection_union_countable} 
$\GeometricProblemSolutions{E,\Lambda,\misurapeso}$ is closed under countable decreasing  intersection and countable increasing union.

\end{enumerate}
\end{corollary}

\subsection{Existence for geometric \texorpdfstring{$L^1$}{Lˆ1}-data}

The following existence result for~\eqref{Pb:GeometricProblem} is a simple consequence of  \cref{prop:existenceL1Datum} and \cref{cor:BINequivalenceFuncGeomPBs}\ref{item:BINfromPtoGP}.

\begin{corollary}[Existence for~\eqref{Pb:GeometricProblem} with geometric $L^1$ datum]
\label{res:existence_geom}
Let \eqref{H:Far_from_zero_int}
and
\eqref{H:Not_integrable}
be in force.
If $E\in\measurablesets$ is such that $|E|<+\infty$, then $\GeometricProblemSolutions{E,\Lambda,\nu}\ne\emptyset$.
\end{corollary}

\subsection{Bounded data and maximal and minimal solutions}

The following result is a simple consequence of \cref{res:intersection_convex}.

\begin{lemma}[Bounded geometric datum]
\label{res:bounded_g_datum}
Let \eqref{H:Radial},
\eqref{H:Not_too_singular}, 
and 
\eqref{H:Decreasing_q} with $q=1$ 
be in force. 
If $E\subset B_R$ for some $R>0$, then also $U\subset B_R$ for every  $U\in\GeometricProblemSolutions{E,\Lambda,\misurapeso}$. 
\end{lemma}

In the following result, we prove the existence of a maximal and a minimal solution of problem~\eqref{Pb:GeometricProblem}, with respect to set inclusion, whenever $E$ has finite measure.

\begin{proposition}[Existence of maximal and minimal solutions]
\label{res:existence_min_max}
Let 
\eqref{H:Far_from_zero_int}
and
\eqref{H:Not_integrable}
be in force.
If $E\in\measurablesets$ with $\min\left\{|E|,|\comp{E}|\right\}<+\infty$, then~\eqref{Pb:GeometricProblem} admits a minimal and a maximal solution $E^-_\Lambda,E^+_\Lambda\in\GeometricProblemSolutions{E,\Lambda,\misurapeso}$ (with respect to inclusion) which are uniquely determined up to $\lebdim$-negligible sets.
Moreover, $E^-_\Lambda$ and $E^+_\Lambda$ satisfy
\begin{equation}\label{eq:min_max_complement}
(E^c)^-_\Lambda=(E^+_\Lambda)^c,
\quad
(E^c)^+_\Lambda=(E^-_\Lambda)^c,
\end{equation}
and
\begin{equation}\label{eq:diamole_un_nome}
	\misurapeso(E^-_\Lambda)\leq\misurapeso(E^+_\Lambda)\leq 2\misurapeso(E).
\end{equation}
\end{proposition}

\begin{proof}
The proof is similar to that of~\cite{CMP15}*{Prop.~6.1} (also see~\cite{B22}*{Lem.~4.7}), so we only sketch it.
Assume $|E|<+\infty$.
By \cref{res:existence_geom}, $\GeometricProblemSolutions{E,\Lambda,\misurapeso}\ne\emptyset$, so take $U\in\GeometricProblemSolutions{E,\Lambda,\misurapeso}$.
By testing the minimality of $U$ against $\emptyset$, we get
\begin{equation*}
\Pker(U)+\Lambda\,\misurapeso(U\bigtriangleup E)
\le 
\Lambda\,\misurapeso(E),
\end{equation*}
which leads to
\begin{equation}\label{eq:stima_misura}
	\misurapeso(U)\leq \misurapeso(U\bigtriangleup E) + \misurapeso(E)\leq 2\misurapeso(E).
\end{equation}
Now let us set	
\begin{equation*}
m=\inf
\big\{\misurapeso(U):U\in\GeometricProblemSolutions{E,\Lambda,\misurapeso}
\big\}
\in
\left[0,2\misurapeso(E)\right]
\end{equation*}
and let $(U_k)_{k\in\N}\subset\GeometricProblemSolutions{E,\Lambda,\misurapeso}$ be such that $\misurapeso(U_k)\to m$ as $k\to+\infty$.
Letting
$E^-_\Lambda
=
\bigcap\limits_{j\in\N}\bigcap\limits_{k=1}^j U_k$ and making use of \cref{res:GP_intersection_union}, we infer that
$E^-_\Lambda
\in
\GeometricProblemSolutions{E,\Lambda,\misurapeso}$
is the minimal solution of~\eqref{Pb:GeometricProblem}.
The construction of $E^+_\Lambda$ is similar and thus left to the reader.
For~\eqref{eq:diamole_un_nome}, we just need to apply~\eqref{eq:stima_misura} to $E^+_\Lambda \supset E^-_\Lambda$ .
In the case $|\comp{E}|<+\infty$, the existence of $E^\pm_\Lambda$ follows from  \cref{res:basic_props_GSol}\ref{item:complement} and the previous case, while \eqref{eq:diamole_un_nome} becomes trivial.
Properties~\eqref{eq:min_max_complement} follow from
\cref{res:basic_props_GSol}\ref{item:complement}.
\end{proof}

The following result provides a comparison principle between maximal and minimal solutions of~\eqref{Pb:GeometricProblem} as $E\in\measurablesets$ varies.
For a strictly related discussion, see the questions left open in~\cite{CE05}*{pp.~1826--1827} and~\cite{YGO07}*{Th.~3.1}.
In \cref{res:comparison} we assume~\eqref{H:Positive} for the first time, see \cref{rem:small_scales} for more details. 

\begin{theorem}[Comparison Principle]
\label{res:comparison}
Let
\eqref{H:Symmetric},
\eqref{H:Far_from_zero_int},
\eqref{H:Not_integrable}
and
\eqref{H:Positive} 
be in force.
Let   $E_1,E_2\in\measurablesets$ be such that 
$\Pker(E_i)<+\infty$ and $
\min\left\{|E_i|,|E_i^c|\right\}<+\infty$, $i=1,2$.
If $E_2\subset E_1$, then 
$(E_2)^-_\Lambda\subset(E_1)^-_\Lambda$ and $(E_2)^+_\Lambda\subset(E_1)^+_\Lambda$.
\end{theorem}

\begin{proof}
Thanks to \cref{res:basic_props_GSol}\ref{item:complement} and \cref{res:existence_geom}, pick $U_1\in\GeometricProblemSolutions{E_1,\Lambda,\misurapeso}$ and $U_2\in\GeometricProblemSolutions{E_2,\Lambda,\misurapeso}$.
Arguing as in~\cite{B22}*{Lem.~4.11}, by minimality of $U_1,U_2$ and~\eqref{eq:submodularity},
\begin{equation}
\label{eq:alpha1_equality}
\Pker(U_1)
+
\Pker(U_2)
=
\Pker(U_1\cap U_2)
+
\Pker(U_1\cup U_2)
\end{equation}
and 
\begin{equation}
\label{eq:alpha2_equality}
\misurapeso\tonde{U_1\bigtriangleup E_1}
+
\misurapeso\tonde{U_2\bigtriangleup E_2}
=
\misurapeso\big((U_1\cup U_2)\bigtriangleup E_1\big)+\misurapeso\big((U_1\cap U_2)\bigtriangleup E_2\big).
\end{equation}
Setting
\begin{equation}
\label{eq:def_interaction_meas}
\interaction(x,y)
=
\kernel(x-y)\,\leb^{2n}(x,y)
\quad
\text{for}\ x,y\in\Rdim,
\end{equation}
using the definition of $\Pker$ and the finite additivity of $\interaction$, equality~\eqref{eq:alpha1_equality} rewrites as
\begin{equation}
\label{eq:punto_cruciale_positivita}
\interaction\big((U_1\setminus U_2)\times (U_2\setminus U_1)\big)=0.
\end{equation}
Since  $\leb^{2\dimension}\ll\interaction$ by~\eqref{H:Positive}, 
we conclude that
\begin{equation}\label{eq:disjunction} 
U_1\subset U_2\ 
\text{or}\ 
U_2\subset U_1.
\end{equation}
Arguing again as in~\cite{B22}*{Lem.~4.11} we deduce that
\begin{equation}
\label{eq:scambio}
U_1\subset U_2
\implies
U_1\in\GeometricProblemSolutions{E_2,\Lambda,\misurapeso},\
U_2\in\GeometricProblemSolutions{E_1,\Lambda,\misurapeso}.
\end{equation}
By \cref{res:existence_min_max}, if $U_1=(E_1)^-_\Lambda$ and $U_2=(E_2)^-_\Lambda$, then $(E_2)^-_\Lambda\subset (E_1)^-_\Lambda$ or $(E_1)^-_\Lambda\subset (E_2)^-_\Lambda$  by~\eqref{eq:disjunction}. 
If $(E_2)^-_\Lambda\subset (E_1)^-_\Lambda$, then we have nothing to prove.
If $(E_1)^-_\Lambda\subset (E_2)^-_\Lambda$ instead, then~\eqref{eq:scambio} readily gives 
$(E_1)^-_\Lambda\in\GeometricProblemSolutions{E_2,\Lambda,\misurapeso}$ 
and
$(E_2)^-_\Lambda\in\GeometricProblemSolutions{E_1,\Lambda,\misurapeso}$,
so that $(E_1)^-_\Lambda=(E_2)^-_\Lambda$. 
The proof that $(E_2)^+_\Lambda\subset (E_1)^+_\Lambda$ is similar and thus left to the reader.
\end{proof}

\begin{remark}[\cref{res:comparison} at small scales]
\label{rem:small_scales}
In the proof of \cref{res:comparison}, \eqref{H:Positive} is used only to pass from~\eqref{eq:punto_cruciale_positivita} to~\eqref{eq:disjunction}.
In other words, \eqref{H:Positive} can be dropped if one is able to ensure that $\interaction$ defined in~\eqref{eq:def_interaction_meas} satisfies
\begin{equation*}
\interaction\big((U_1\setminus U_2)\times (U_2\setminus U_1)\big)=0
\implies
U_1\subset U_2\ 
\text{or}\ 
U_2\subset U_1
\end{equation*}
whenever $U_i\in\GeometricProblemSolutions{E_i,\Lambda,\nu}$, $i=1,2$.
The above implication holds if  
\begin{equation}
\label{strettina}
K(x-y)>0\
\text{for a.e.}\ 
x\in U_1\setminus U_2\ \text{and}\ y\in U_2\setminus U_1
\end{equation}
whenever $U_i\in\GeometricProblemSolutions{E_i,\Lambda,\nu}$, $i=1,2$.
Thanks to \cref{res:pos}, assuming \eqref{H:Decreasing_q}  and~\eqref{H:Doubling}, we have $\supp K\supset B_{4D}$.
If $D<+\infty$ (otherwise \eqref{H:Positive} is satisfied), then~\eqref{strettina} holds provided that $U_1,U_2\subset B_D$ for $U_i\in\GeometricProblemSolutions{E_i,\Lambda,\nu}$, $i=1,2$.
By \cref{res:bounded_g_datum}, an thus assuming~\eqref{H:Decreasing_q} with $q=1$, this holds by additionally requiring that 
\begin{equation}
\label{scatola}
E_2\subset E_1\subset B_D
\end{equation}
Hypothesis~\eqref{scatola} provides a \emph{small-scale version} of  \cref{res:comparison}. 
\end{remark}

\subsection{Convex data and uniqueness outside the jump set}

For $f\in\Lspace{1}(\Rdim)$, we define
\begin{equation*}
\begin{split}
\distSolDatPLUS(f,\Lambda,\misurapeso)
&=
\sup\set*{\|u-f\|_{\Lspace{1}(\Rdim,\,\misurapeso)}:u\in\FunctionalProblemSolutions{f,\Lambda,\misurapeso}},
\\[.5ex]
\distSolDatMINUS(f,\Lambda,\misurapeso)
&=
\inf\set*{\|u-f\|_{\Lspace{1}(\Rdim,\,\misurapeso)}:u\in\FunctionalProblemSolutions{f,\Lambda,\misurapeso}},
\end{split}
\end{equation*}
for all $\Lambda>0$.
We also define the \emph{jump set}
\begin{equation*}		\FunctionalProblemJumps{f,\misurapeso}
=
\set{\Lambda>0:\distSolDatMINUS(f, \Lambda,\misurapeso)<\distSolDatPLUS(f, \Lambda,\misurapeso)}
\end{equation*}
associated to~\eqref{Pb:FunctionalProblem}.
The following result generalizes~\cite{B22}*{Lem.~4.3}.
Its proof follows the same line of that of~\cite{CE05}*{Claim~5} and is thus omitted.

\begin{lemma}[Monotonicity of $\distSolDatPLUSMINUS(f,\Lambda,\misurapeso)$]
If $f\in\Lspace{1}(\Rdim)$, then 
\begin{equation*}
\distSolDatMINUS(f,\Lambda_1,\misurapeso) \leq \distSolDatPLUS(f,\Lambda_1,\misurapeso) \leq \distSolDatMINUS(f,\Lambda_2,\misurapeso) \leq \distSolDatPLUS(f,\Lambda_2,\misurapeso)
\end{equation*}
for all $\Lambda_1\geq\Lambda_2>0$.
Consequently, $\Lambda\mapsto\distSolDatPLUSMINUS(f,\Lambda,\misurapeso)$ are two decreasing functions whose sets of discontinuity points contain  $\FunctionalProblemJumps{f,\misurapeso}$.
In particular, $\FunctionalProblemJumps{f,\misurapeso}$ is a countable set. 
\end{lemma}

As a consequence, we get the following result, generalizing~\cite{B22}*{Th.~4.4}.

\begin{theorem}[Uniqueness outside the jump set]\label{res:uniqueness_convex}
Let 
\eqref{H:Radial},
\eqref{H:Not_too_singular},
\eqref{H:Not_integrable}
and
\eqref{H:Decreasing_q}
with $q=1$ be in force.
If $E\subset\Rdim$ is a bounded convex set, then 
$\FunctionalProblemSolutions{\chi_{E},\Lambda,\misurapeso}=\set*{\chi_{U_\Lambda}}$
for all $\Lambda\in(0,+\infty)\setminus\FunctionalProblemJumps{\chi_E,\misurapeso}$,
where $U_\Lambda\in\measurablesets$ is such that $U_\Lambda\subset E$.
\end{theorem}

\begin{proof}
Fix $\Lambda\in (0,+\infty)\setminus \FunctionalProblemJumps{\chi_{E},\misurapeso}$.
By \cref{prop:existenceL1Datum}, $\FunctionalProblemSolutions{\chi_{E},\Lambda,\misurapeso}\ne\emptyset$.
Since $E$ is convex, the conclusion follows by \cref{res:intersection_convex} as in the proof of \cite{B22}*{Th. 4.4}.
\end{proof}

\section{Fidelity analysis}
\label{sec:fidelity}

In this section, we study the behavior of the solutions of the functional and geometric $\kernel$-variation denoising models with respect to the fidelity parameter.
Here we let
\begin{equation}
\label{eq:def_below_above_peso}
\belowpeso=\essinf_{\Rdim}w,
\quad
\abovepeso=\esssup_{\Rdim}w,
\end{equation}
for any given $\nu=w\lebdim\in\MisureAmmissibili$
and notice that $0<\belowpeso\le\abovepeso<+\infty$ by~\eqref{eq:def_misure_ammissibili}.

\subsection{High fidelity}

The following result generalizes~\cite{B22}*{Th.~4.5}.

\begin{theorem}[High fidelity for $C^{1,1}$ regular sets]\label{res:high_fidelity_for_sets}
Let 
\eqref{H:Radial},
\eqref{H:Not_too_singular},
\eqref{H:Not_integrable},
\eqref{H:Positive}
and
\eqref{H:Decreasing_q} with $q=1$ 
be in force.
If $E\subset\Rdim$ is an open set of class~$C^{1,1}$ with $\min\{|E|,|E^c|\}<+\infty$, then there exists $\Lambda_{E,\kernel,\belowpeso}>0$, depending on~$E$, $\kernel$ and~$\belowpeso$ only, such that 
$\GeometricProblemSolutions{E,\Lambda,\misurapeso}=\set*{E}$ and 
$\GeometricProblemSolutions{E^c,\Lambda,\misurapeso}=\set*{E^c}$ 
for $\Lambda\ge\Lambda_{E,\kernel,\belowpeso}$.
\end{theorem}

To prove \cref{res:high_fidelity_for_sets}, we start with the case $E$ is a ball, generalizing~\cite{B22}*{Lem.~4.13}.

\begin{proposition}[High fidelity for balls]
\label{res:high_fidelity_for_balls}
Let 
\eqref{H:Radial},
\eqref{H:Not_too_singular},
\eqref{H:Not_integrable}
and
\eqref{H:Decreasing_q} with $q=1$ 
be in force.
If $x\in\Rdim$ and $r>0$, then
$\GeometricProblemSolutions{B_r(x),\Lambda,\misurapeso}
=
\set*{B_r(x)}$
and
$\GeometricProblemSolutions{B_r(x)^c,\Lambda,\misurapeso}
=
\set*{B_r(x)^c}$
for $\Lambda\ge\Lambda_{r,\kernel,\belowpeso}=\frac{2\,\Pker(B_r)}{\belowpeso\,|B_r|}$.
\end{proposition}

\begin{proof}
Fix $x\in\Rdim$ and $r>0$.
By \cref{res:existence_geom}, $\GeometricProblemSolutions{B_r(x),\Lambda,\nu}\ne\emptyset$, so pick $U\in\GeometricProblemSolutions{B_r(x),\Lambda,\nu}$. 
By \cref{res:bounded_g_datum},  $U\subset B_r(x)$ and thus $B^{|U|}(x)\subset B_r(x)$, where $B^{|U|}(x)=x+B^{|U|}$ (recall~\eqref{eq:ballification} for the latter symbol), so that
\begin{equation}
\label{panza}
|U\bigtriangleup B_r(x)|
=
|B_r(x)|-|U|
=
|B_r(x)|-|B^{|U|}(x)|.
\end{equation} 
Therefore, by \cref{res:isoperimetric}, \eqref{eq:P_K_translation} and~\eqref{panza}, we can estimate
\begin{equation}
\label{fischio1}
\begin{split}
\Pker(B_r(x))
&\ge 
\Pker(U)+\Lambda\nu\big(U\bigtriangleup B_r(x)\big)
\ge
\Pker(B^{|U|})+\Lambda\belowpeso\,|U\bigtriangleup B_r(x)|
\\
&=
\Pker(B^{|U|}(x))+\Lambda\belowpeso\,\big(|B_r(x)|-|B^{|U|}(x)|\big).
\end{split}
\end{equation}
Now, for $\Lambda\belowpeso\ge\frac{2\,\Pker(B_r(x))}{|B_r(x)|}=\frac{2\,\Pker(B_r)}{|B_r|}$ and $|B^{|U|}(x)|<|B_r(x)|$, we can estimate
\begin{equation}
\label{fischio2}
\begin{split}
\Pker(B^{|U|}(x))+\Lambda\belowpeso&\,\big(|B_r(x)|
-|B^{|U|}(x)|\big)
=
\Pker(B^{|U|}(x))+\Lambda\belowpeso\,\frac{|B_r(x)|^2-|B^{|U|}(x)|^2}{|B_r(x)|+|B^{|U|}(x)|}
\\
&>
\Pker(B^{|U|}(x))+\Lambda\belowpeso\,\frac{|B_r(x)|^2-|B^{|U|}(x)|^2}{2|B_r(x)|}
\\
&\ge
\Pker(B^{|U|}(x))+\frac{\Pker(B_r(x))}{|B_r(x)|^2}\,\big(|B_r(x)|^2-|B^{|U|}(x)|^2\big).
\end{split}
\end{equation}
By \cref{res:monotonicity_isop_ratio}, we must have
\begin{equation}
\label{fischio3}
\Pker(B^{|U|}(x))+\frac{\Pker(B_r(x))}{|B_r(x)|^2}\,\big(|B_r(x)|^2-|B^{|U|}(x)|^2\big)
\ge 
\Pker(B_r(x))
\end{equation}
and thus, by combining~\eqref{fischio1}, \eqref{fischio2} and~\eqref{fischio3}, we get that 
\begin{equation*}
\Pker(B_r(x))
>
\Pker(B^{|U|}(x))+\frac{\Pker(B_r(x))}{|B_r(x)|^2}\,\big(|B_r(x)|^2-|B^{|U|}(x)|^2\big)
\ge 
\Pker(B_r(x))
\end{equation*}
whenever $|B^{|U|}(x)|<|B_r(x)|$.
This is clearly a contradiction, so $|B^{|U|}(x)|=|B_r(x)|$, from which $U=B_r(x)$.
The conclusion hence follows from  \cref{res:basic_props_GSol}\ref{item:complement}.
\end{proof}

\begin{remark}[On the constant $\Lambda_{r,\kernel,\belowpeso}$]
\label{rem:valerio_estimate}
By~\eqref{H:Not_too_singular}, one can apply~\eqref{eq:valerio_set} to $B_r$ and get
\begin{equation*}
\Lambda_{r,\kernel,\belowpeso}=\frac{2\,\Pker(B_r)}{\belowpeso\,|B_r|}
\le
\left(2\vee \frac\dimension{r}\right)
\frac{1}{\belowpeso}\int_{\Rdim}(1\wedge|x|)\,K(x)\,\de x
<+\infty
\quad
\text{for}\ r>0.
\end{equation*}
In the case of the fractional perimeter $P_s$ (see~\cite{B22} for the definition), $P_s(B_r)=P_s(B_1)\,r^{\dimension-s}$, so 
$\Lambda_{r,K_s,\belowpeso}=\frac{2P_s(B_1)}{\belowpeso\,r^s}$ as in~\cite{B22}*{Lem.~4.13} (up to multiplicative constants).
However, the proof of \cref{res:high_fidelity_for_balls} differs from the one of~\cite{B22}*{Lem.~4.13}, since $\Pker$ does not enjoy any scaling property, explaining while here we relied on \cref{res:monotonicity_isop_ratio}.
Also note that we do not need the characterization of equality in \cref{res:isoperimetric}, again differently from~\cite{B22}*{Lem.~4.3} 
\end{remark}

\begin{proof}[Proof of~\cref{res:high_fidelity_for_sets}]
Since $\partial E$ is of class $C^{1,1}$, we find $r_0=r_0(E)>0$ and two countable families $\mathscr F_{\rm int}
=
\set*{B_{s_k}(x_k) : x_k\in E,\ s_k\ge r_0}$ and
$\mathscr F_{\rm ext}
=
\set*{B_{t_k}(y_k) : y_k\in E^c,\ t_k\ge r_0}$
with
\begin{equation}
\label{eq:inscatolato}
B_{s_k}(x_k)\subset E
\subset
B_{t_k}(y_k)^c,
\quad
\bigcup_{k\in\N}B_{s_k}(x_k)=E,
\quad
\bigcap_{k\in\N} B_{t_k}(y_k)^c=\closure{E}.
\end{equation}
We hence define
\begin{equation}
\label{eq:rock}
\Lambda_{E,\kernel,\belowpeso}
=
\frac{2\,\Pker(B_{r_0})}{\belowpeso\,|B_{r_0}|}
\end{equation}
and note that $\Lambda_{E,\kernel,\belowpeso}<+\infty$ by~\eqref{H:Not_too_singular} as in \cref{rem:valerio_estimate}.
Let $\Lambda\ge\Lambda_{E,\kernel,\belowpeso}$ and note that 
\begin{equation}
\label{eq:combo}
\GeometricProblemSolutions{B_{s_k}(x_k),\Lambda,\misurapeso}=\set*{B_{s_k}(x_k)},
\quad
\GeometricProblemSolutions{B_{t_k}(y_k)^c,\Lambda,\misurapeso}=\set*{B_{t_k}(y_k)^c},
\end{equation}
for all $k\in\N$, thanks to \cref{res:high_fidelity_for_balls}.
Since either $|E|$ or $|E^c|$ is finite, by \cref{res:existence_min_max} we can find a minimal and a maximal solution $E_\Lambda^-,E_\Lambda^+\in\GeometricProblemSolutions{E,\Lambda,\misurapeso}$ which are uniquely determined up to $\lebdim$-negligible sets. 
Moreover, since either $|E|<+\infty$ or $|E^c|<+\infty$, and $\partial E$ is of class $C^{1,1}$, either $E$ or $E^c$ is bounded, so that $\Pker(E)<+\infty$ because of~\eqref{eq:valerio_embed}, being either $\chi_E\in BV(\Rdim)$ or $\chi_{E^c}\in BV(\Rdim)$.  
Hence, \cref{res:comparison}, in combination with~\eqref{eq:inscatolato} and~\eqref{eq:combo}, implies that 
\begin{equation}
\label{eq:inscatolate_min_max}
B_{s_k}(x_k)
\subset 
E^-_\Lambda
\subset
E^+_\Lambda
\subset 
B_{t_k}(y_k)^c
\end{equation}	
for all $k\in\N$.
By~\eqref{eq:inscatolato} and~\eqref{eq:inscatolate_min_max} we get
$E\subset E^-_\Lambda\subset E^+_\Lambda\subset E$
up to $\lebdim$-negligible sets.
\end{proof}

\begin{definition}[Function with $C^{1,1}$ regular superlevel sets]
\label{def:funz_regular}
A function~$f$ has \emph{uniformly $C^{1,1}$ regular superlevel sets} if there exists $r_0=r_0(f)>0$ such that, for each $t\in\R$, $E_t=\set*{f>t}$, we can find two countable families
$\mathscr F_{\rm int}(t)
=
\set*{B_{s_k}(x_k) : x_k\in E_t,\ s_k\ge r_0}$,
$\mathscr F_{\rm ext}(t)
=
\set*{B_{t_k}(y_k) : y_k\in E^c_t,\ t_k\ge r_0}$,
such that
\begin{equation*}
B_{s_k}(x_k)\subset E_t
\subset
B_{t_k}(y_k)^c,
\quad
\bigcup_{k\in\N}B_{s_k}(x_k)=E_t,
\quad
\bigcap_{k\in\N} B_{t_k}(y_k)^c=\closure{E_t}.
\end{equation*}
\end{definition}

If $f\in\Lspace{1}(\Rdim)$ has uniformly $C^{1,1}$ regular (and thus open, in particular) superlevel sets $E_t=\set*{f>t}$, $t\in\R$, then  either $E_t$ or $(E_t)^c$ is bounded for  $t\ne0$. 

The following result  generalizes~\cite{B22}*{Th.~4.16} and can be seen as a non-local counterpart of~\cite{CE05}*{Th.~5.6}, which is instead proved via a calibration argument.
We omit its proof.

\begin{corollary}[High fidelity for regular $\Lspace{1}$ functions]\label{res:high_fidelity_for_functions}
Let
\eqref{H:Positive},
\eqref{H:Radial},
\eqref{H:Not_too_singular},
\eqref{H:Not_integrable}
and
\eqref{H:Decreasing_q} with $q=1$ 
be in force.
If $f\in\Lspace{1}(\Rdim)$ is as in \cref{def:funz_regular}, then there exists $\Lambda_{f,\kernel,\belowpeso}>0$, depending on~$f$, $\kernel$ and~$\belowpeso$ only, such that 
$\FunctionalProblemSolutions{f,\Lambda,\misurapeso}=\set*{f}$
for $\Lambda\ge\Lambda_{f,\kernel,\belowpeso}$.
\end{corollary}

\subsection{Low fidelity}

The following result generalizes~\cite{B22}*{Th.~4.18}. 

\begin{theorem}[Low fidelity]\label{res:low_fidelity}
Let 
\eqref{H:Radial},
\eqref{H:Not_too_singular},
\eqref{H:Not_integrable},
\eqref{H:Decreasing_q} with $q=1$ 
and 
\eqref{H:Doubling} 
be in force.
Given $R<\frac D4$, there exists $\Lambda_{R,\kernel,\abovepeso}>0$, depending on~$R$, $\kernel$ and~$\abovepeso$ only, such that,
if $f\in\Lspace{1}(\Rdim)$ with $\supp f\subset B_R$, then 
$\FunctionalProblemSolutions{f,\Lambda,\misurapeso}
=
\set*{0}$
for $\Lambda<\Lambda_{R,\kernel,\abovepeso}$.
\end{theorem}

\begin{proof}
Let $\Lambda>0$.
Since $f\in\Lspace{1}(\R^n)$, $\FunctionalProblemSolutions{f,\Lambda,\misurapeso}\ne\emptyset$ by \cref{prop:existenceL1Datum}, so pick $u\in\FunctionalProblemSolutions{f,\Lambda,\misurapeso}$.
We claim that $u=0$. 
Since $u^\pm\in\FunctionalProblemSolutions{f^\pm,\Lambda,\misurapeso}$ by\cref{res:basic_props_Sol}\ref{item:truncation}, we can assume $f\ge0$.
By \cref{res:basic_props_GSol}\ref{item:fromPtoGP}, $\set*{u>t}\in\GeometricProblemSolutions{\set*{f>t},\Lambda,\misurapeso}$ for $t>0$.
Since $\set*{f>t}\subset\supp f\subset B_R$ for $t>0$, by \cref{res:bounded_g_datum} we also get $\set*{u>t}\subset B_R$ for $t>0$, so that $\supp u\subset\overline{B_R}$.
Thus, by testing the minimality of~$u$ against  $v=0$, we get
\begin{equation*}
\kerTV{u}+\Lambda\|u-f\|_{\Lspace{1}(B_R,\,\misurapeso)}
\le
\Lambda\|f\|_{\Lspace{1}(B_R,\,\misurapeso)}.
\end{equation*} 
In addition, by \cref{res:lusin}, we have
\begin{equation*}
C_\kernel\kerTV{u}
\ge 
\phi_\kernel(2|h|,D)
\,
\|u(\cdot+h)-u\|_{\Lspace{1}(\Rdim)}
\end{equation*} 
for $h\in\Rdim$ with $|h|\le\frac D2$, where $C_\kernel>0$ depends on~$\kernel$ only.
On the other side,  we have
\begin{equation*}
\|u(\cdot+h)-u\|_{\Lspace{1}(\Rdim)}
=
2\,\|u\|_{\Lspace{1}(B_R)}\geq\frac{2}{\abovepeso}\,\|u\|_{\Lspace{1}(B_R,\,\misurapeso)}
\end{equation*}
for $h\in\Rdim$ with $|h|\ge2R$. 
Therefore, for $h\in\Rdim$ with $|h|\in\left[2R,\frac D2\right]$, we get that
\begin{equation*}
\frac{2\,\phi_\kernel(2|h|,D)}{\abovepeso\,C_\kernel}
\,
\|u\|_{\Lspace{1}(B_R,\misurapeso)}
+
\Lambda\|u-f\|_{\Lspace{1}(B_R,\misurapeso)}
\le
\Lambda\|f\|_{\Lspace{1}(B_R,\misurapeso)}
\end{equation*}
and thus
\begin{equation*}
\left(\frac{2\,\phi_\kernel(2|h|,D)}{\abovepeso\,C_\kernel}-\Lambda\right)
\,
\|u\|_{\Lspace{1}(B_R,\misurapeso)}
\le
\Lambda
\big(
\|f\|_{\Lspace{1}(B_R,\misurapeso)}
-
\|u-f\|_{\Lspace{1}(B_R,\misurapeso)}
-
\|u\|_{\Lspace{1}(B_R,\misurapeso)}
\big)
\le 0.
\end{equation*}
The conclusion hence follows by choosing 
\begin{equation}
\label{eq:low_constant}
\Lambda_{R,\kernel,\abovepeso}
=
\sup_{h\in B_{D/2}\setminus B_{2R}}
\frac{2\,\phi_\kernel(2|h|,D)}{\abovepeso\,C_\kernel}
=
\frac{2\,\phi_\kernel(4R,D)}{\abovepeso\,C_\kernel}
\in(0,+\infty)
\end{equation}
and the proof is complete.
\end{proof}

\begin{remark}[On the constant $\Lambda_{R,\kernel,\abovepeso}$]
\label{rem:constant_in_low_fidelity}
The proof of \cref{res:low_fidelity} differs from the one of~\cite{B22}*{Th.~4.18}, since we cannot directly rely on \cref{res:sobolev}, due to the implicit  function~$\beta_\kernel$.
Nonetheless, in the fractional case, the definition in~\eqref{eq:low_constant} gives 
\begin{equation*}
\Lambda_{R,\kernel_s,\abovepeso}
=
\frac{2}{\abovepeso\,C_\kernel}
\int_{\Rdim\setminus B_{4R}}|x|^{-\dimension-s}\,\de x
=
\frac{2\dimension|B_1|}{s4^s\,\abovepeso\,C_\kernel}\,\frac1{R^s}
\quad
\text{for}\ R>0,
\end{equation*}
(note that, in this case, \eqref{H:Doubling} holds with $D=+\infty$), which is the bound found in~\cite{B22}*{Th.~4.18} (up to multiplicative constants).
However, the strategy of proof of~\cite{B22}*{Th.~4.18} can be adapted to \cref{res:low_fidelity} via \cref{res:sobolev_finite_supp} provided that  \eqref{H:Decreasing_q}  holds with~$q\in[n,n+1)$, with no need of assuming~\eqref{H:Doubling} in this case.
\end{remark}

\subsection{\texorpdfstring{$(\kernel,\misurapeso)$}{(K,ν)}-Cheeger sets and fidelity}

The following result refines \cref{res:uniqueness_convex} if the datum $E$ is bounded and convex (and, possibly, calibrable).

\begin{theorem}[Relation with fidelity]
\label{res:relation_with_fidelity}
Let 
\eqref{H:Radial},
\eqref{H:Not_too_singular},
\eqref{H:Not_integrable},
\eqref{H:Decreasing_q}
with $q\in[1,\dimension+1)$
and 
\eqref{H:Doubling}
with $D=+\infty$
be in force.
If $E\subset\Rdim$ is bounded and convex, with non-empty interior, then:
\begin{enumerate}[label=(\roman*)]

\item\label{item:lambda_zero_cheeger}
$\Cheegkerpeso(E)
=
\sup\set*{\Lambda>0
:
\emptyset\in\GeometricProblemSolutions{E,\Lambda,\misurapeso}
}\in(0,+\infty)$;

\item\label{item:sotto_cheeger} 
if $\Lambda<\Cheegkerpeso(E)$, then $\GeometricProblemSolutions{E,\Lambda,\misurapeso}=\set*{\emptyset}$;

\item\label{item:uguale_cheeger} 
if $\Lambda=\Cheegkerpeso(E)$, then 
$\GeometricProblemSolutions{E,\Lambda,\misurapeso}=\CheegerkerpesoClass(E)\cup\set*{\emptyset}$ and so
\begin{equation}
\label{eq:uguale_cheeger}
\begin{split}
&\FunctionalProblemSolutions{\chi_{E},\Cheegkerpeso(E),\misurapeso}
\\
&\qquad=\set*{
u\in\BVkerspace(\Rdim;[0,1]):\{u>t\} \in \CheegerkerpesoClass(E)\cup\set*{\emptyset}\ 
\text{for all}\ 
t\in [0,1)
};
\end{split}
\end{equation}
	
\item\label{item:sopra_Cheeger} 
if $\Lambda>\Cheegkerpeso(E)$ and the set~$E$ is $(\kernel,\misurapeso)$-calibrable, then $\GeometricProblemSolutions{E,\Lambda,\misurapeso}=\set*{E}$.

\end{enumerate}
\end{theorem}

\begin{proof}
Since $E$ is bounded and convex with non-empty interior, $E$ is $\kernel$-admissible by \eqref{H:Not_too_singular}, so $\Cheegkerpeso(E)\in(0,+\infty)$.
Since $E$ is convex, $\Pker(U\cap E)
\le
\Pker(U)
$ for $U\in\measurablesets$ with $|U|<+\infty$ by \cref{res:intersection_convex}.
Moreover, $\misurapeso\big((U\cap E)\bigtriangleup E\big)
\le 
\misurapeso\tonde{U\bigtriangleup E}$
for $U\in\measurablesets$.
Hence~\eqref{Pb:GeometricProblem} is equivalent to the following minimization problem
\begin{equation*}
\inf\big\{\Pker(U)-\Lambda\misurapeso(U) 
:
U\in\measurablesets,\ U\subset E
\big\}.
\end{equation*}
Let us set
$\Lambda_0(E,\misurapeso)
=
\sup\set*{\Lambda>0
:
\emptyset\in\GeometricProblemSolutions{E,\Lambda,\misurapeso}
}$.
Since $E$ is bounded, $\Lambda_0(E,\misurapeso)\in[0,+\infty)$ by \cref{cor:BINequivalenceFuncGeomPBs}\ref{item:BINfromGPtoP} and \cref{res:low_fidelity}. 
As in the proof of~\cite{B22}*{Th.~4.21}, we get $\Lambda_0(E,\misurapeso)=\Cheegkerpeso(E)$, proving~\ref{item:lambda_zero_cheeger}.
Points \ref{item:sotto_cheeger}, \ref{item:uguale_cheeger}, \ref{item:sopra_Cheeger} can be proved as in~\cite{B22}*{Th.~4.21 and Th~4.22}, so we leave the details to the reader.
\end{proof}

\begin{remark}[\cref{res:relation_with_fidelity} at small scales]
\label{rem:small-scale_relation_fidelity_Cheeger}
If we only require that \eqref{H:Doubling} holds with $D<+\infty$ in \cref{res:relation_with_fidelity}, then we have to assume that $\closure{E}\subset B_{D/4}$ in order to ensure that $\Lambda_0(E,\misurapeso)<+\infty$ (recall \cref{res:low_fidelity}).
This yields a \emph{small-scale version} of \cref{res:relation_with_fidelity}.
\end{remark}

\begin{remark}[\cref{res:relation_with_fidelity} under~\eqref{H:Decreasing_q} for $q\in[\dimension,\dimension+1)$]
\label{rem:no_doubling_in_relation_fidelity}
By \cref{rem:constant_in_low_fidelity}, we can drop~\eqref{H:Doubling} in \cref{res:relation_with_fidelity} (and in \cref{rem:small-scale_relation_fidelity_Cheeger}) if~\eqref{H:Decreasing_q} holds with $q\in[n,n+1)$.
\end{remark}

The following result is a consequence of \cref{res:relation_with_fidelity}, \cref{rem:no_doubling_in_relation_fidelity} and \cref{res:calibrable}, and refines \cref{res:high_fidelity_for_balls} and \cref{res:low_fidelity} (applied to $f=\chi_B$) in the case $\misurapeso=\lebdim$ (also improving the constants given by \cref{res:high_fidelity_for_sets} and \cref{res:high_fidelity_for_functions}).

\begin{corollary}[Fidelity for balls]
\label{res:exact_fidelity_ball}
Let
\eqref{H:Radial},
\eqref{H:Not_too_singular},
\eqref{H:Not_integrable}
and
\eqref{H:Decreasing_dim_strinct}
be in force.
If $B$ is a (non-trivial) ball, then
\begin{equation*}
\GeometricProblemSolutions{B,\Lambda}
=
\begin{cases}
\set*{\emptyset}
&
\text{for}\ \Lambda<\frac{\Pker(B)}{|B|},
\\[3mm]
\set*{\emptyset,B}
&
\text{for}\
\Lambda=\frac{\Pker(B)}{|B|},
\\[3mm]
\set*{B}
&
\text{for}\ \Lambda>\frac{\Pker(B)}{|B|}.
\end{cases}
\end{equation*}
\end{corollary}


\end{document}